\documentclass[article]{amsart}

\usepackage{amssymb,amsfonts,amsmath,amsthm}
\usepackage[all,arc]{xy}
\usepackage{enumerate}
\usepackage{mathrsfs}
\usepackage[toc,page]{appendix}
\usepackage[left=3cm, right=3cm, bottom=3cm]{geometry}
\usepackage{graphicx}
\usepackage{tabularx}
\usepackage{url}
\usepackage{color}
\usepackage{tikz-cd}  
\usepackage{mathdots} 
\usepackage{romannum} 
\usepackage{hyperref} 

\newtheorem{thm}{Theorem}[section]

\newtheorem*{thm*}{Theorem}
\newtheorem*{cor*}{Corollary}
\newtheorem*{prop*}{Proposition}
\newtheorem{cor}[thm]{Corollary}
\newtheorem{prop}[thm]{Proposition}
\newtheorem{lem}[thm]{Lemma}

\theoremstyle{definition}
\newtheorem{defn}[thm]{Definition}

\newtheorem{exmp}[thm]{Example}

\newtheorem*{notn*}{Notation}

\theoremstyle{remark}
\newtheorem{rem}[thm]{Remark}

\newtheorem*{idea*}{Idea}

\newcommand{\Spec}{{\rm Spec}}

\makeatletter
\let\c@equation\c@thm
\makeatother
\numberwithin{thm}{section}
\numberwithin{equation}{section}

\title[Logahoric Higgs torsors for a complex reductive group]{Logahoric Higgs torsors for a complex reductive group}

\author{Georgios Kydonakis, Hao Sun and Lutian Zhao}

\begin{document}
\pagenumbering{arabic}
\maketitle
\begin{abstract}
	In this article, a logahoric Higgs torsor is defined as a parahoric torsor with a logarithmic Higgs field. For a connected complex reductive group $G$, we introduce a notion of stability for logahoric $\mathcal{G}_{\boldsymbol\theta}$-Higgs torsors on a smooth algebraic curve $X$, where $\mathcal{G}_{\boldsymbol\theta}$ is a parahoric group scheme on $X$. In the case when the group $G$ is the general linear group ${\rm GL}_n$, we show that the stability condition of a parahoric torsor is equivalent to the stability of a parabolic bundle. A correspondence between semistable logahoric $\mathcal{G}_{\boldsymbol\theta}$-Higgs torsors and semistable equivariant logarithmic $G$-Higgs bundles allows us to construct the moduli space explicitly. This moduli space is shown to be equipped with an algebraic Poisson structure.
\end{abstract}

\flushbottom



\renewcommand{\thefootnote}{\fnsymbol{footnote}}
\footnotetext[2]{MSC2020 Class: 14D23, 32Q26 (Primary), 14L15, 53D17 (Secondary)}

\vspace{3mm}
\begin{flushright}
	\textit{To the memory of Professor M. S. Narasimhan}
\end{flushright}
\vspace{3mm}

\section{Introduction}

For a compact Riemann surface $X$ of genus $g \ge 2$, stable vector bundles on $X$ of fixed rank and degree can be characterized in terms of irreducible unitary representations of a certain discrete group. This is the content of the theorem of Narasimhan and Seshadri \cite{NS}, which provides a correspondence between such stable bundles and irreducible representations at the level of moduli spaces.

In search of a natural generalization of this remarkable result to the case when $X$ is noncompact, Seshadri introduced in \cite{Seshadri} an additional layer of structure on the bundles over a smooth irreducible projective curve, which he called a \emph{parabolic structure}, inspired by the work of Weil on logarithmic connections with regular singularity at finitely many points \cite[\S 2, Chapter I]{Weil}. This notion involved the choice of a weighted flag on the fiber over each point from a finite collection of points on the curve. The new objects were called by Seshadri \emph{parabolic bundles} and a stability condition in terms of a parabolic bundle degree was introduced analogously to the considerations of Mumford in the absence of the parabolic structure; for this notion of stability, the Narasimhan--Seshadri correspondence was subsequently established by Mehta and Seshadri \cite{MS} in this open-curve context, involving fundamental group representations into the group $G= \text{U}(n)$.

The next important step was to extend this correspondence for compact, as well as for non-compact groups $G$. The case when $G= {\rm GL}_n(\mathbb{C})$ carried out by Simpson \cite{Simp} was a landmark in this direction and involved the introduction of filtered objects to clarify the correct version of the bijective correspondence; in particular, stable filtered regular Higgs bundles and stable filtered local systems.

The main objective in the present article is to introduce a notion of stability for Higgs pairs and then construct a Dolbeault moduli space using this stability condition for general complex reductive groups $G$, in which moduli space the Higgs pairs generalize the stable filtered regular Higgs bundles of Simpson in the parabolic setting. The language to be used will be that of \emph{parahoric group schemes} in the sense of Bruhat--Tits \cite{BT1, BT2} for the notion of parahoric weight introduced by Boalch \cite{Bo}. This moduli space will be moreover shown to be algebraic Poisson. Before we explain our main considerations leading to the definition of these stable Higgs pairs for the general groups $G$, it is instructive to review a number of approaches in the literature followed for this problem. Several ideas from these approaches have been adapted in our work.

\subsection{Background} In generalizing the notion of a stable parabolic vector bundle to the setting of principal $G$-bundles for semisimple or reductive structure groups $G$, a central problem that soon became apparent was to introduce the correct notion of a parabolic weight in order to get a moduli space and a bijective correspondence, which would coincide with the ones of Simpson when $G= \text{GL}_n(\mathbb{C} )$. Early attempts in this direction were using a rather insufficient definition for a parabolic $G$-bundle; see for instance \cite{Hit90} or \cite{LS}. In \cite{BBN}, Balaji, Biswas and Nagaraj looked at principal bundles from a Tannakian perspective \cite{DM}, following the description given by Nori \cite{No1, No2}. In this sense, principal $G$-bundles are interpreted as functors from the category of locally free coherent sheaves, and a functor in the parabolic context serves as the right definition that respects the tensor product operation. Even though this definition coincides with the one of Seshadri when $G= \text{GL}_n(\mathbb{C} )$, it became clear that to a representation of a Fuchsian group into the maximal compact subgroup of $G$ with fixed holonomy around each puncture will not correspond a principal $G$-bundle in general.

An alternative approach by Teleman and Woodward \cite{TeWo} involved switching the order of embedding the group $G$ in $\text{GL}_n(\mathbb{C} )$ and applying the equivalence with equivariant bundles. The weights in this case, called markings, were defined to lie in a Weyl alcove for the corresponding Lie algebra; this meant though, that one should restrict to a subclass of parabolic $G$-bundles in order to establish an analog of the Mehta-Seshadri correspondence.

In this same line of an approach, Biquard, Garc\'{i}a-Prada and Mundet i Riera \cite{BGM} defined a notion of weight for parabolic principal $G$-bundles and proved a correspondence in the case of a real reductive group (also including the complex group cases) by using a Donaldson functional and the existence of harmonic reductions in this setting. The notion of weight in \cite{BGM} involves a choice, for each point in the reduced effective divisor $D \subset X$, of an element in a Weyl alcove $\mathcal{A}$ of the Lie algebra of a fixed maximal torus in a fixed maximal compact subgroup of a non-compact reductive Lie group, with the closure of $\mathcal{A}$ containing 0. The authors allow these elements to lie in a wall of the Weyl alcove, in order to establish the correspondence with parabolic $G$-connections having arbitrary fixed holonomy around the points in $D$. The side-effect of this explicit approach is that under this definition for a parabolic principal $G$-bundle, to a connection corresponds not a single holomorphic bundle, but rather a class of holomorphic bundles equivalent under gauge transformations with meromorphic singularities.

This defect started to become clear through the work of Boalch, who first defined in \cite{Bo} the notion of weight from the point of view of \emph{parahoric torsors} instead. Parahoric torsors were introduced by Pappas and Rapoport in \cite{PR1}, and locally these are described as parahoric subgroups of a formal loop group in the sense of Bruhat and Tits \cite{BT1, BT2}. Several conjectures concerning the moduli space of parahoric torsors were made by Pappas and Rapoport in \cite{PR2}, most of which have been verified by Heinloth in \cite{He}, thus generalizing corresponding results by Drinfeld and Simpson \cite{DS} for the moduli stack of principal bundles on a smooth projective curve over an arbitrary field. Then, a weight for a parahoric torsor is a point of the corresponding Bruhat--Tits building in the facet corresponding to the parahoric subgroup of a formal loop group (Definition 1, p. 46 in \cite{Bo}).

In an independent work but still in this approach, Balaji and Seshadri \cite{BS} introduced a notion of stability for parahoric torsors for a collection of weights chosen from the set of rational one-parameter subgroups of the group $G$, for $G$ semisimple and simply connected over $\mathbb{C}$. In this case, parahoric torsors on a smooth complex projective curve $X$ of genus $g \ge 2$ are indeed the correct intrinsically defined objects on $X$ associated to a $\left( \pi ,G \right)$-bundle on the upper half plane $\mathbb{H}$, where $\pi $ is the subgroup of the discontinuous group of automorphisms of $\mathbb{H}$, such that $X={\mathbb{H}}/{\pi }$. Balaji and Seshadri moreover constructed in \textit{loc. cit.} a moduli space for this notion of stability of parahoric torsors on $X$ and proved, under the assumption that the weights are rational, an analogue of the Mehta--Seshadri correspondence in this context, that is, for the case $G= \text{U}(n)$. Subsequently, Balaji, Biswas and Pandey extended in \cite{BBP} the correspondence to the case of real weights using a definition of stability that covers real weights as well.

\subsection{Results}
In this article we introduce a stability condition for logahoric Higgs torsors (also parahoric torsors) and construct the moduli space of (semi)stable logahoric Higgs torsors for general complex reductive groups $G$. We give in the sequel more details about the precise statements.

Let $X$ be a smooth algebraic curve with a reduced effective divisor $D$. Denote by $K_X$ the canonical line bundle on $X$. Let $G$ be a connected complex reductive group. Fixing a maximal torus $T$ in $G$, we equip each point $x \in D$ with a rational weight $\theta_x \in Y(T) \otimes_{\mathbb{Z}} \mathbb{Q}$, where $Y(T)$ is the group of one-parameter subgroups of $T$. Denote by ${\boldsymbol\theta}:=\{\theta_x, x \in D\}$ the collection of weights over the points in $D$. A parahoric (Bruhat-Tits) group scheme $\mathcal{G}_{\boldsymbol\theta}$ is then defined by gluing local parahoric group schemes for formal disks around each point $x \in D$ (see Definition \ref{202}). We then define the following pairs as the Higgs version of Boalch's ``\emph{tame parahoric connections}'' (\cite[\S 2.3]{Bo}):
\begin{defn}
A \emph{logahoric $\mathcal{G}_{\boldsymbol\theta}$-Higgs torsor} on $X$ is defined as a pair $(E,\varphi)$, where
\begin{itemize}
	\item $E$ is a parahoric $\mathcal{G}_{\boldsymbol\theta}$-torsor on $X$;
	\item $\varphi \in H^0(X, E(\mathfrak{g}) \otimes K_X(D))$ is a section, where $E(\mathfrak{g})$ is the adjoint bundle of $E$.
\end{itemize}
The section $\varphi$ is called a \emph{logarithmic Higgs field}. 
\end{defn}

Note that the definition of logahoric Higgs torsors is a slightly modified version of Yun's definition \cite[\S 4.3]{Yun}, where the Higgs field $\varphi$ is considered as a section of $E(\mathfrak{g})(D)$. Moreover, Baraglia, Kamgarpour and Varma in \cite{BKV} use a similar definition for a parahoric Higgs bundle for a semisimple and simply connected Lie group $G$, where the Higgs field is considered as an element in $H^0(X, E(\mathfrak{g})^{*} \otimes K_X)$ with nilpotent residue, where $E(\mathfrak{g})^{*}$ denotes the dual bundle.  The reason why we choose to work with $K_X(D)$ instead of $L=\mathcal{O}(D)$ is that Higgs bundles are degenerate connections, and one can define meromorphic connections, that is, corresponding to $K_X(D)$ with $D \ge 0$  (but one cannot define``$L$-twisted connections" for arbitrary line bundles $L$); therefore this seems to be a right setup towards establishing a non-abelian Hodge correspondence in this logahoric setting for a connected complex reductive group.

Balaji and Seshadri showed in \cite{BS} that parahoric $\mathcal{G}_{\boldsymbol\theta}$-torsors on $X$ are equivalent to $\Gamma$-equivariant $G$-principal bundles (called \emph{$(\Gamma,G)$-bundles} in this article) on $Y$, where $Y \rightarrow X$ is a Galois covering with Galois group $\Gamma$. Based on this work, we generalize the correspondence to Higgs bundles:
\begin{thm}[Theorem \ref{305}]
	Let $\mathcal{M}_H(X,\mathcal{G}_{\boldsymbol\theta})$ be the stack of logahoric $\mathcal{G}_{\boldsymbol\theta}$-Higgs torsors, and let $\mathcal{M}^{\boldsymbol\rho}_H(Y,\Gamma,G)$ be the stack of $\Gamma$-equivariant logarithmic $G$-Higgs bundles of type ${\boldsymbol\rho}$, where ${\boldsymbol\rho}$ is a fixed set of representations $\{\rho_y: \Gamma_y \rightarrow T, y \in R\}$, for a set of points $R$ of $Y$. Then we have an isomorphism
	\begin{align*}
	\mathcal{M}^{\boldsymbol\rho}_H(Y,\Gamma,G) \cong \mathcal{M}_H(X,\mathcal{G}_{\boldsymbol\theta})
	\end{align*}
	as algebraic stacks.
\end{thm}

The notion of stability that we introduce for the logahoric Higgs torsors described above, is inspired by the works of Ramanathan (\cite{Rama1975,Rama19961,Rama19962}) on the construction of moduli spaces of semistable principal $G$-bundles on a smooth projective irreducible complex curve. For an appropriate notion of a \emph{parahoric degree} of a parahoric $\mathcal{G}_{\boldsymbol\theta}$-torsor $E$ (see Definition \ref{401}), the definition of this Ramanathan-stability is the following:
\begin{defn}[Definition \ref{403}]
	A parahoric $\mathcal{G}_{\boldsymbol\theta}$-torsor $E$ is called \emph{$R$-stable} (resp. \emph{$R$-semistable}), if for
	\begin{itemize}
		\item any proper parabolic group $P \subseteq G$,
		\item any reduction of structure group $\varsigma: X \rightarrow E/\mathcal{P}_{\boldsymbol\theta}$ ($\mathcal{P}_{\boldsymbol\theta}$ is a parahoric group constructed from $P$),
		\item any nontrivial anti-dominant character $\chi: \mathcal{P}_{\boldsymbol\theta} \rightarrow \mathbb{G}_m$, which is trivial on the center of $\mathcal{P}_{\boldsymbol\theta}$,
	\end{itemize}
	one has
	\begin{align*}
	parh\deg E(\varsigma,\chi) > 0, \quad (\text{resp. } \geq 0).
	\end{align*}
\end{defn}
\noindent An important remark to make here is that when one considers small weights, this definition coincides with the one of Balaji and Seshadri in \cite[\S 6]{BS}, and parahoric $\mathcal{G}_{\boldsymbol\theta}$-torsors in this case are precisely parabolic bundles. The notion of Ramanathan-stability for a logahoric $\mathcal{G}_{\boldsymbol\theta}$-Higgs torsor $(E,\varphi)$ now assumes the compatibility of the Higgs field $\varphi$ (Definition \ref{405}). With respect to the correspondence between logahoric Higgs torsors and equivariant logarithmic Higgs bundles, we prove that this correspondence also holds under stability conditions.

\begin{thm}[Theorem \ref{416}]
	Let $(E,\varphi)$ be a logahoric $\mathcal{G}_{\boldsymbol\theta}$-Higgs torsor on $X$, and let $(F,\phi)$ be the corresponding $\Gamma$-equivariant logarithmic $G$-Higgs bundle on $Y$. Then, $(E,\varphi)$ is $R$-stable (resp. $R$-semistable) if and only if $(F,\phi)$ is $R$-stable (resp. $R$-semistable).
\end{thm}
\noindent This is also the key property to construct the moduli space of logahoric $\mathcal{G}_{\boldsymbol\theta}$-Higgs torsors, as it implies that constructing this moduli space is equivalent to rather constructing the moduli space of $R$-stable (or $R$-semistable) $\Gamma$-equivariant logarithmic $G$-Higgs bundles.

In \S\ref{sect_R_mu_stab}, we study further the $R$-stability condition introduced above. Since we work on general complex reductive groups, the center is generally non-trivial. In the definition of $R$-stability, we require that the anti-dominant character acts trivially on the center (see \S\ref{sect_stab_cond}). However, this condition weakened, in particular not requiring that the anti-dominant character is acting trivially on the center, provides a modified notion of stability which we call \emph{$R_\mu$-stability} (Definitions \ref{504} and \ref{Ralpha}). In the case when the group $G$ is  ${\rm GL} (n, \mathbb{C} )$ and for an appropriate choice of $\mu$, we show that this $R_\mu$-stability condition is equivalent to the stability condition for a parabolic Higgs bundle as considered by Simpson in \cite{Simp}. Furthermore, we prove the following relation between $R$-stability and $R_\mu$-stability:

\begin{prop}[Proposition \ref{alphaexst}]
	Let $E$ be a parahoric $\mathcal{G}_{\boldsymbol\theta}$-torsor. There exists a canonical choice of $\mu\in \mathfrak{z}$, where $\mathfrak{z}$ is the center of $\mathfrak{t}$, depending on the topological type of $E$, such that $E$ is $R$-stable (resp. $R$-semistable) if and only if $E$ is $R_\mu$-stable (resp. $R_\mu$-semistable).
\end{prop}

\noindent The element $\mu$ is regarded as a topological invariant of parahoric $\mathcal{G}_{\boldsymbol\theta}$-torsors. We say that a parahoric $\mathcal{G}_{\boldsymbol\theta}$-torsor is of \emph{type $\mu$}, if $\mu$ is given from the above proposition. 

\begin{thm}[Theorem \ref{701}]
	There exists a quasi-projective scheme $\mathfrak{M}^{Rss}_H(X,\mathcal{G}_{\boldsymbol\theta},\mu)$ as the moduli space for the moduli problem $\mathcal{M}^{Rss}_H(X,\mathcal{G}_{\boldsymbol\theta},\mu)$ of $R$-semistable logahoric $\mathcal{G}_{\boldsymbol\theta}$-Higgs torsors, and the geometric points of $\mathfrak{M}^{Rss}_H(X,\mathcal{G}_{\boldsymbol\theta},\mu)$ represent $S$-equivalence classes of $R$-semistable logahoric $\mathcal{G}_{\boldsymbol\theta}$-Higgs torsors of type $\mu$. Furthermore, there is an open subset $\mathfrak{M}^{Rs}_H(X,\mathcal{G}_{\boldsymbol\theta},\mu) \subseteq \mathfrak{M}^{Rss}_H(X,\mathcal{G}_{\boldsymbol\theta},\mu)$ parameterizing isomorphism classes of $R$-stable logahoric $\mathcal{G}_{\boldsymbol\theta}$-Higgs torsors of type $\mu$.
\end{thm}

In fact, one gets the same moduli space when considering irrational weights. For any irrational weight there is an element in the rational apartment such that for every quasi-parahoric torsor the (semi)stability conditions coincide (Remark \ref{irrational}).

We finally show that the moduli space of $R$-semistable logahoric Higgs torsors is equipped with an algebraic Poisson structure (Proposition \ref{801}). This follows the strategy developed in \cite{KSZ4} and involves the construction of an Atiyah sequence inducing a Lie algebroid structure on the tangent space of the moduli space of $R$-stable $\Gamma$-equivariant logarithmic $G$-Higgs bundles. 

\subsection{Applications}
We close this introduction with a discussion about the possible applications and further directions in which the considerations of this article may evolve. In this article, we construct the moduli space of logahoric Higgs torsors with respect to a given complex reductive group. For the case when the group $G$ is a real reductive group, an analogous approach as in \cite{BGM} but for the algebraic moduli space can be used in order to obtain the construction of the moduli space also in this case. Also, this approach can be applied to construct the moduli space of logahoric connections (see Remark \ref{moduli of local system}), which is a special case of a $\Lambda$-module (see \cite{Simp2} or Definition \ref{703}).


Secondly, as has already been pointed out in \cite[\S 6]{Bo}, the Dolbeault moduli space of semistable logahoric Higgs torsors for general complex reductive groups is expected to provide the correct setup in order to establish a bijective correspondence extending the correspondence of Simpson \cite{Simp}, and the existence of this moduli space is given in this paper. Furthermore, in \cite[\S 6]{Bo} the author includes a table describing the correspondence of the parameters involved in the correspondence, namely the parahoric weights and the eigenvalues of the Higgs field on the one hand, and the weights and monodromy of a logarithmic connection on the other hand. The Riemann--Hilbert correspondence for logahoric connections established in \cite{Bo} provides the description of the corresponding Betti data as the $G$-version of the ``\emph{$\mathbb{R}$-filtered local systems}'' in Simpson's work for the $\text{GL}_n(\mathbb{C})$-case \cite{Simp}; we refer the reader to \cite[Remark 2]{Bo} for this description, as well as to \cite[\S 6]{BGM} when referring to the parabolic situation. Therefore, we believe that the moduli space of semistable logahoric Higgs torsors constructed here is the correct choice in order to establish the \emph{tame parahoric nonabelian Hodge correspondence} for general complex reductive groups.

We moreover expect that the logahoric case treated in this article can be used in order to construct algebraically the moduli space of logahoric Higgs torsors in the case of irregular singularities, thus referring to wild character varieties and the description of corresponding Stokes data (cf. \cite{Bo2,Bo3,Bo4}). This space has been constructed analytically in the case when $G=\text{GL}_n(\mathbb{C})$ in \cite{BiBo}, and a nonabelian Hodge correspondence for this moduli space was established combining results from \cite{BiBo} and \cite{Sabbah}. The construction of moduli spaces of logahoric Higgs torsors for arbitrary complex groups is important also from the point of view of understanding the tamely ramified geometric Langlands correspondence as proposed in the work of Gukov and Witten \cite{GukWit1, GukWit2}. Namely, it is argued that the category of A-branes is equivalent to the derived category of coherent sheaves on the moduli stack of parabolic $^{L}G$-local systems, while the category of B-branes is equivalent to the derived category of $D$-modules on the moduli stack of parabolic $G$-bundles.


\vspace{3mm}

\begin{notn*} Throughout the article, we will be distinguishing the notation between the parahoric Higgs bundles and equivariant Higgs bundles as follows:
	
	\vspace{2mm}
	
	\begin{tabularx}{\textwidth}{XXX}
		& Parahoric & Equivariant \\
		\hline
		Curve: & $X$ & $Y$ \\
		Local coordinate: & $z$ & $w$ \\
		Coordinate Ring: & $A$ & $B$ \\
		Function field: & $K$ & $L$ \\
		Torsor/Bundle: & $E$ & $F$ \\
		Higgs field: & $\varphi$ & $\phi$ \\
		Reduction of structure group: & $\varsigma$ & $\sigma$ \\
		Character: & $\kappa$ & $\chi$ \\
		\hline
	\end{tabularx}
	\vspace{2mm}
\end{notn*}

\section{Parahoric Torsors and Equivariant Bundles}

The notion of a parahoric subgroup is similar to that of a parabolic subgroup and can be described using the theory of affine buildings, also known as Bruhat--Tits buildings, originally developed by Bruhat and Tits in their series of articles \cite{BT1,BT2} (see also  \cite{Tits,Tits2,Weiss} for surveys on the structure of affine buildings). The word \textit{parahoric} is a blend word between the words ``parabolic'' and ``Iwahori''. An Iwahori subgroup is a subgroup of a reductive algebraic group over a non-archimedian local field, analogous to Borel subgroups of an algebraic group. The seminal work of Bruhat and Tits \textit{loc. cit.} is extending to the case of reductive algebraic groups over a local field the study of Iwahori and Matsumoto \cite{IM} on the Iwahori subgroups for the Chevalley groups over $p$-adic fields.

Parahoric group schemes $\mathcal{G}$ and parahoric $\mathcal{G}$-torsors on a smooth projective curve were introduced by Pappas and Rapoport in \cite{PR1}. The notion of weight for such torsors was first defined by Boalch in \cite{Bo}, while Balaji and Seshadri in \cite{BS} introduced a notion of stability for parahoric $\mathcal{G}$-torsors for a collection of small weights chosen from the set of rational one-parameter subgroups of the group $G$, assuming that $G$ is semisimple and simply connected.

In this section, we generalize the setup of Balaji and Seshadri. We set the definition for a parahoric torsor over a complex reductive group with a collection of (arbitrary) rational weights following \cite{BS} and \cite{Bo}, and see that parahoric torsors correspond to $\Gamma$-equivariant $G$-bundles, similarly to \cite{BS}.

\subsection{Parahoric Group Schemes and Parahoric Torsors}
Let $G$ be a connected complex reductive group. We fix a maximal torus $T$ in $G$. Let $X(T):={\rm Hom}(T,\mathbb{G}_m)$ be the character group and $Y(T):={\rm Hom}(\mathbb{G}_m,T)$ be the group of one-parameter subgroups of $T$. The group $Y(T)$ can be understood as a lattice of the Lie algebra $\mathfrak{t}$ of $T$. Let
\begin{align*}
\langle \cdot, \cdot \rangle: Y(T) \times X(T) \rightarrow \mathbb{Z}
\end{align*}
be the canonical pairing, that extends to $\mathbb{Q}$ by tensoring $Y(T)$ with $\mathbb{Q}$.

We will denote by $R$, the root system with respect to the maximal torus $T$. Given a root $r \in R$, there is a root homomorphism
\begin{align*}
{\rm Lie}(\mathbb{G}_a) \rightarrow ({\rm Lie}(G))_r.
\end{align*}
This isomorphism induces a natural homomorphism
\begin{align*}
u_r: \mathbb{G}_a \rightarrow G,
\end{align*}
such that $t u_r(a) t^{-1}= u_r(r(t)a)$ for $t \in T$ and $a \in \mathbb{G}_a$. Denote by $U_r$ the image of the homomorphism $u_r$, which is a closed subgroup.

A \emph{(rational) weight} $\theta$ is an element in $Y(T) \otimes_{\mathbb{Z}} \mathbb{Q}$. Under differentiation, we can consider $\theta$ as an element in $\mathfrak{t}$. Define the integer $m_r(\theta):=\lceil -r(\theta) \rceil$, where $\lceil \cdot \rceil$ is the ceiling function and $r(\theta):=\langle \theta, r \rangle$. We then introduce the following:

\begin{defn}\label{parahoric_subgroup}
	Let $A:=\mathbb{C}[[z]]$ and $K:=\mathbb{C}((z))$. With respect to the above data, we define the \emph{parahoric subgroup $G_{\theta}$} of $G(K)$ as
	\begin{align*}
	G_{\theta}:=\langle T(A), U_r(z^{m_r(\theta)}A), r \in R \rangle.
	\end{align*}
	Denote by $\mathcal{G}_{\theta}$ the corresponding group scheme of $G_\theta$, which is called the \emph{parahoric group scheme}.
\end{defn}

\begin{rem}
	An equivalent analytic definition of the parahoric group $G_\theta$ was given in \cite[\S 2.1]{Bo} as
	\begin{align*}
	G_{\theta}:=\{g \in G(K) \text{ }|\text{ } z^{\theta} g z^{-\theta} \text{ has a limit as $z \rightarrow 0$ along any ray}\},
	\end{align*}
	where $z^{\theta}:= {\rm exp}(\theta {\rm log}(z))$.
\end{rem}

A weight is called \emph{small}, if $r(\theta) < 1$ for all roots $r \in R$, and Balaji and Seshadri studied parahoric torsors in this situation in \cite{BS}. Furthermore, if we assume that $G$ is semisimple, given any weight $\theta$, the parahoric subgroup $G_\theta$ is conjugate to a parahoric subgroup $G_{\theta_0}$, where $\theta_0$ satisfies the condition that for any $r \in R$, $\theta_0(r) \leq 1$ and the conjugation is taken in $G(K)$ (see \cite{Tits2}, Section 3.1, p. 50). With respect to this property, in this paper, we work on weights $\theta$ such that for any $r \in R$, we have $\theta(r) \leq 1$.

\begin{rem}\label{201}
	Let $\{t_i\}$ be the generators of $Y(T)$. By abuse of notation, we regard $\{t_i\}$ as a basis of $\mathfrak{t}$. Then, a rational weight $\theta$ (regarded as the corresponding element in $\mathfrak{t}$) can be written as $\theta=\sum \frac{a_i}{d_i} t_i$, where $a_i$ and $d_i$ are integers. Denote by $d$ the least common multiple of $d_i$. We can assume that the denominators in the coefficients of $t_i$ are equal to $d$. This integer $d$ corresponds to the order of cyclic group $\Gamma$ when we discuss the correspondence in \S 2.3.
\end{rem}

The above construction is a local picture of parahoric group schemes. Now we will define the parahoric group schemes globally. Let $X$ be a smooth algebraic curve over $\mathbb{C}$, and we also fix a reduced effective divisor $D$ on $X$. In fact, the divisor $D$ is a sum of $s$ many distinct points. For each point $x \in D$, we equip it with a rational weight $\theta_x \in Y(T) \otimes_{\mathbb{Z}} \mathbb{Q}$. Denote by ${\boldsymbol\theta}:=\{\theta_x, x \in D\}$ the collection of weights over the points in $D$.

\begin{defn}\label{202}
	Let ${\boldsymbol\theta}$ be a collection of weights over $D$. We define a group scheme $\mathcal{G}_{\boldsymbol\theta}$ over $X$ by gluing the following local data
	\begin{align*}
	\mathcal{G}_{\boldsymbol\theta}|_{X\backslash D} \cong G \times (X\backslash D), \quad \mathcal{G}_{\boldsymbol\theta}|_{\mathbb{D}_x} \cong \mathcal{G}_{\theta_x}, x \in D,
	\end{align*}
	where $\mathbb{D}_x$ is a formal disc around $x$. This group scheme $\mathcal{G}_{\boldsymbol\theta}$ will be called a \emph{parahoric Bruhat--Tits group scheme}.
\end{defn}
By \cite[Lemma 3.18]{ChGP}, the group scheme $\mathcal{G}_{\boldsymbol\theta}$ defined above is a smooth affine group scheme of finite type, flat over $X$. We will denote by ${\rm Bun}(X,\mathcal{G}_{\boldsymbol\theta})$, the category of parahoric $\mathcal{G}_{\boldsymbol\theta}$-torsors $E$ on $X$.

\begin{lem}[Corollary 4.2.6 in \cite{Yun}]\label{203}
	The category of parahoric $\mathcal{G}_{\boldsymbol\theta}$-torsors ${\rm Bun}(X,\mathcal{G}_{\boldsymbol\theta})$ has a natural stack structure. More precisely, ${\rm Bun}(X,\mathcal{G}_{\boldsymbol\theta})$ is an algebraic stack locally of finite type.
\end{lem}

\subsection{Equivariant Bundles}\label{subsect_equi_bundle}
Let $\Gamma$ be a cyclic group of order $d$ with generator $\gamma$, and let $G$ be a connected complex reductive group. Define $B:=\mathbb{C}[[w]]$ and $L:=\mathbb{C}((w))$. There is a natural $\Gamma$-action on $\mathbb{D}:=\Spec(B)$, such that $\gamma w=\xi w$, where $\xi$ is a $d$-th root of unity. We first study the local picture of a $(\Gamma,G)$-bundle over $\mathbb{D}$.

\begin{defn}\label{204}
	A \emph{$\Gamma$-equivariant $G$-bundle} over $\mathbb{D}$ is a $G$-bundle together with a lift of the action of $\Gamma$ on the total space of $F$. A $\Gamma$-equivariant $G$-bundle is also called a \emph{$(\Gamma,G)$-bundle}.
\end{defn}

Since we work on an affine chart, a $G$-bundle $F$ has the property that
\begin{align*}
F(\mathbb{D}) \cong G(B).
\end{align*}
Therefore, a $(\Gamma,G)$-bundle $F$ is equivalent to a representation $\rho: \Gamma \rightarrow G$. Note that $\Gamma$ is a cyclic group. We can suppose that the representation $\rho$ factors through $T$ under a suitable conjugation. Then, a representation $\rho: \Gamma \rightarrow T$ corresponds to an element in $Y(T)$ with order $d$, that is,
\begin{align*}
{\rm Hom}(\Gamma, T) \cong {\rm Hom}(X(T),X(\Gamma))={\rm Hom}(X(T),\mathbb{Z}/d\mathbb{Z})=Y(T)/d\cdot Y(T).
\end{align*}
Therefore, a weight $\theta \in Y(T) \otimes_{\mathbb{Z}} \mathbb{Q}$ uniquely determines a representation $\rho: \Gamma \rightarrow T$. 

Let $Y$ be a smooth algebraic curve over $\mathbb{C}$ equipped with a $\Gamma$-action.
\begin{defn}\label{205}
	A \emph{$(\Gamma,G)$-bundle} on $Y$ is a $G$-bundle $F$ together with a lift of the action of $\Gamma$ on the total space of $F$, which preserves the action of the group $G$.
\end{defn}
We now return to the local picture of a $(\Gamma,G)$-bundle on $Y$. Given $y \in Y$, let $\Gamma_y$ be the stabilizer group of the point $y$. Denote by $R$ the set of points in $Y$, of which the stabilizer groups are nontrivial. As was discussed above, the $\Gamma$-action around $y \in R$ is given by a representation $\rho_y: \Gamma_y \rightarrow T$,  such that
\begin{align*}
\gamma \cdot (u,g) \rightarrow (\gamma u, \rho_y(\gamma) g), \quad u \in \mathbb{D}_y, \gamma \in \Gamma_y,
\end{align*}
where $\mathbb{D}_y$ is a $\Gamma$-invariant formal disc around $y$.

\begin{defn}\label{equivariant_bundle_of_type} We say that a $(\Gamma,G)$-bundle $F$ is \emph{of type ${\boldsymbol\rho}=\{\rho_y, y \in R\}$}, if the representation for each $y \in R$ is given by $\rho_y$. Denote by ${\rm Bun}^{\boldsymbol\rho}(Y,\Gamma,G)$ the category of $(\Gamma,G)$-bundles of type $\boldsymbol\rho$ on $Y$.
\end{defn}
A $(\Gamma,G)$-bundle $F$ on $Y$ can be also understood from gluing the following local data. For each $y \in R$, we define $F_y:= \mathbb{D}_y  \times G$, such that the $\Gamma_y$-action is defined as
\begin{align*}
\gamma \cdot (u,g) \rightarrow (\gamma u, \rho_y(\gamma)g), \quad u \in \mathbb{D}_y, \gamma \in \Gamma_y,
\end{align*}
and define $F_0:=(Y \backslash R) \times G$ with the $\Gamma_y$-structure
\begin{align*}
\gamma \cdot (u,g) \rightarrow (\gamma u, g), \quad u \in Y \backslash R, \gamma \in \Gamma_y.
\end{align*}
Therefore, a $(\Gamma,G)$-torsor $F$ being of type $\boldsymbol\rho$, is equivalent to giving $(\Gamma,G)$-isomorphisms
\begin{align*}
\Theta_y: F_y|_{\mathbb{D}^{\times}_y} \rightarrow F_0|_{\mathbb{D}^{\times}_y}, \quad y \in R.
\end{align*}
Note that given two transition functions $\Theta_1$ and $\Theta_2$, if there exist $(\Gamma,G)$-isomorphisms
\begin{align*}
\tau_y: F_y \rightarrow F_y, \quad \tau_0: F_0 \rightarrow F_0,
\end{align*}
such that $\Theta_1=\tau_0 \Theta_2 \tau_y$, then the corresponding $(\Gamma,G)$-bundles are isomorphic. The above observation gives us the following theoretic isomorphism \cite[Proposition 3.1.1]{BS}
\begin{align*}
{\rm Bun}^{\boldsymbol\rho}(Y,\Gamma,G) \cong [ \prod_{y \in R} G(B) \backslash \prod_{y \in R} G(L) / G(\mathbb{C}[Y \backslash R])  ]^{\Gamma}.
\end{align*}

On the other hand, given a reduced effective divisor $R \subseteq Y$ and a collection of cyclic groups $\{\Gamma_y| y \in R\}$, there is a canonical way to define a root stack $\mathcal{X}$ (see \cite{Cad} for more details). Therefore, $(\Gamma,G)$-bundles of type $\boldsymbol\rho$ on $Y$ are equivalent to $G$-bundles of type $\boldsymbol\rho$ on the corresponding root stack $\mathcal{X}$, in other words,
\begin{align*}
{\rm Bun}^{\boldsymbol\rho}(Y,\Gamma,G) \cong {\rm Bun}^{\boldsymbol\rho}(\mathcal{X},G).
\end{align*}
Since $\mathcal{X}$ is a Deligne--Mumford stack, ${\rm Bun}^{\boldsymbol\rho}(\mathcal{X},G)$ has a natural stack structure. Furthermore, it is an algebraic stack locally of finite type by Artin's theorem \cite{Art1974,Hall}. This gives the following lemma.
\begin{lem}\label{206}
	The stack ${\rm Bun}^{\boldsymbol\rho}(Y,\Gamma,G)$ is an algebraic stack locally of finite type.
\end{lem}

\subsection{Correspondence}\label{subsect_corr_torsor_and_bundle}
We first work on local charts. Let $\Gamma$ be a cyclic group of order $d$ with generator $\gamma$, and there is a natural $\Gamma$-action on $B=\mathbb{C}[[w]]$ defined by rotation. Let $\theta$ be a weight with the common denominator $d$. Denote by $\rho$ the corresponding representation. We define 
\begin{align*}
\Delta(w):=w^{d \cdot \theta}.
\end{align*}
Clearly, we have
\begin{align*}
\Delta(\gamma w)=\rho(\gamma) \Delta(w).
\end{align*}

Now we consider a $(\Gamma,G)$-bundle on $B$. Let $F$ be a $(\Gamma,G)$-bundle on $\Spec (B)$ of type $\rho$. Denote by $\mathbb{U}:={\rm Aut}_{(\Gamma,G)}(F) \subseteq G(B)$ the automorphism group of $F$. Given an element $\sigma \in \mathbb{U}$, let $\varsigma:=\Delta^{-1} \sigma \Delta$. We then have
\begin{align*}
\varsigma(\gamma w)=\varsigma(w),
\end{align*}
which means that $\varsigma$ is $\Gamma$-invariant. Therefore, it can be descended to an element $G(A)$ by substituting $z=w^d$, where $A=\mathbb{C}[[z]]$. Note that for each root $r \in R$, we have
\begin{align*}
\varsigma(w)_r=\sigma(w)_r w^{-d \cdot r(\theta)},
\end{align*}
where the subscript $r$ means that the element is in $U_r(B)$. Note that $\sigma(w)_r$ is a holomorphic function, and $\varsigma(w)$ is a $\Gamma$-invariant meromorphic function. Substituting $z=w^d$, we have
\begin{align*}
\varsigma(z)_r=\sigma(z)_r z^{-r(\theta)}.
\end{align*}
Therefore, the order of the pole of $\varsigma(z)_r$ is bounded by $\lceil -r(\theta) \rceil$. Since $\varsigma(w)_r$ is $\Gamma$-invariant, we have $\varsigma(z)_r \in U_r(z^{m_r(\theta)}\mathbb{C}[[z]])$ for each $r \in R$. In conclusion, the element $\varsigma(z)$ is in $G_{\theta}$. The above discussion implies the isomorphism
\begin{align*}
\mathbb{U} \cong G_{\theta}.
\end{align*}
This also implies that a $(\Gamma,G)$-bundle $F$ of type $\rho$ over $\Spec(B)$ corresponds to a unique parahoric $\mathcal{G}_{\theta}$-torsor $E$ over $\Spec(A)$, and this is a one-to-one correspondence.

Now we consider the correspondence globally. Let $X$ be a smooth algebraic curve over $\mathbb{C}$ of genus $g \geq 2$ with a fixed reduced effective divisor $D$. We fix a collection of rational weights $\boldsymbol\theta=\{\theta_x, x \in D\}$, where $\theta_x \in Y(T) \otimes_{\mathbb{Z}} \mathbb{Q}$. Denote by $d_x$ the denominator of the rational weight $\theta_x$. The data $(X,D,(d_x)_{x \in D})$ uniquely determine a Galois covering $\pi: Y \rightarrow X$ with Galois group $\Gamma$ such that
\begin{itemize}
	\item $D$ is the branch locus;
	\item $R:=\pi^{-1}(D)$ is the ramification locus;
	\item the stabilizer group of $y=\pi^{-1}(x)$ is $\Gamma_y:=\Gamma_{d_x}$, where $x \in D$ and $\Gamma_{d_x}$ is the cyclic group of order $d_x$.
\end{itemize}

Let $\boldsymbol\rho:=\{\rho_y,y \in R\}$ be a collection of representations $\rho_y: \Gamma_y \rightarrow T$. Given a $(\Gamma,G)$-bundle $F$ on $Y$ of type $\boldsymbol\rho$, the restriction $F_y:=F|_{\mathbb{D}_y}$ is a $(\Gamma,G)$-bundle on a $\Gamma$-invariant formal disc $\mathbb{D}_y$. By the discussion above, the $(\Gamma,G)$-bundle $F_y$ of type $\rho_y$ corresponds to a unique $\mathcal{G}_{\theta_x}$-torsor of type $\theta_x$ on $\mathbb{D}_x$, where $\mathbb{D}_x$ is a formal disc around $x=\pi(y) \in D$ and $\theta_x$ is the rational weight corresponding to $\rho_y$. By gluing the local data $\{(F|_{Y\backslash R})^{\Gamma}, F_y, y \in R \}$ together, we get a $\mathcal{G}_{\boldsymbol\theta}$-torsor $E$ on $X$, where ${\boldsymbol\theta}:=\{\theta_x, x \in D\}$. This correspondence is actually a one-to-one correspondence.
\begin{thm}[Theorem 5.3.1 in \cite{BS}]\label{207}
	With respect to the notation above, there is an isomorphism
	\begin{align*}
	{\rm Bun}^{\boldsymbol\rho}(Y,\Gamma,G) \cong {\rm Bun}(X,\mathcal{G}_{\boldsymbol\theta})
	\end{align*}
	as stacks.
\end{thm}
\noindent This theorem implies that the correspondence also holds as algebraic stacks locally of finite type.

Now we consider the general case. Let $\theta$ be a weight corresponding to the representation $\rho$ as discussed above. Define
\begin{align*}
\delta:=\theta+\vartheta \in Y(T)\otimes_{\mathbb{Z}} \mathbb{Q},
\end{align*}
where $\vartheta$ is a weight such that $\gamma^{\vartheta}=I$ (the identity element in $G$). Let $\Delta_\delta(w):=w^{\delta}$ be an element in $T(B)$. We have
\begin{align*}
\Delta_\delta(\gamma w)=(\gamma w)^{\theta+\vartheta}=\gamma^{\theta} w^{\theta+\vartheta}=\rho(\gamma) \Delta_\delta(w),
\end{align*}
which means that $\Delta_\delta(w)$ is $\Gamma$-equivariant. Applying the same proof as for Theorem \ref{207} to $\Delta_\delta(w)$, we get the following isomorphism
\begin{align*}
{\rm Bun}^{\boldsymbol\rho}(Y,\Gamma,G) \cong {\rm Bun}(X,\mathcal{G}_{\boldsymbol\delta}).
\end{align*}
This observation provides the following proposition.
\begin{prop}\label{208}
	With respect to the notation above, there is an isomorphism of stacks
	\begin{align*}
	{\rm Bun}(X,\mathcal{G}_{\boldsymbol\theta}) \cong {\rm Bun}(X,\mathcal{G}_{\boldsymbol\delta}).
	\end{align*}
\end{prop}
\noindent Note that this isomorphism can be also realized in terms of Hecke transformations; we refer to \cite[\S 3.3]{BGM} for a detailed exposition.

\section{Logahoric Higgs Torsors and Equivariant Logarithmic Higgs bundles}

In this section, we study logahoric Higgs torsors and equivariant logarithmic Higgs bundles. We prove that there is a correspondence between logahoric $\mathcal{G}_{\boldsymbol\theta}$-Higgs torsors on $X$ and $\Gamma$-equivariant logarithmic $G$-Higgs bundles of type $\boldsymbol\rho$ on $Y$, and this correspondence implies the isomorphism of the corresponding algebraic stacks (Theorem \ref{305}).

\subsection{Logahoric Higgs Torsors}
In this subsection, we define logahoric $\mathcal{G}_{\boldsymbol\theta}$-Higgs torsors on smooth algebraic curves $X$ following the notation from \S 2.1.

Let $X$ be a smooth algebraic curve with a given reduced effective divisor $D$. Denote by $K_X$ the canonical line bundle on $X$. Let $G$ be a connected reductive complex group together with a set of weights $\boldsymbol\theta=\{\theta_x, x \in D\}$. Denote by $\mathcal{G}_{\boldsymbol\theta}$ the parahoric group scheme over $X$ and let $E$ be a parahoric $\mathcal{G}_{\boldsymbol\theta}$-torsor on $X$. Note that for a parahoric $\mathcal{G}_{\boldsymbol\theta}$-torsor $E$, we can define its adjoint bundle on each local chart and then glue everything together. Thus, we obtain a bundle and denote it by $E(\mathfrak{g})$ or ${\rm Ad}(E)$.  The following definition is the Higgs version of Boalch's ``\emph{tame parahoric connection}'' (\cite[\S 2.3]{Bo}):

\begin{defn}\label{301}
	A \emph{logahoric $\mathcal{G}_{\boldsymbol\theta}$-Higgs torsor} on $X$ is a pair $(E,\varphi)$, where
	\begin{itemize}
		\item $E$ is a parahoric $\mathcal{G}_{\boldsymbol\theta}$-torsor on $X$;
		\item $\varphi \in H^0(X, E(\mathfrak{g}) \otimes K_X(D))$ is a section.
	\end{itemize}
	The section $\varphi$ is called a \emph{logarithmic Higgs field}.
\end{defn}

\begin{rem}
	In Yun's article \cite[\S 4.3]{Yun}, the logarithmic Higgs field $\varphi$ is considered as a section of $E(\mathfrak{g})(D)$. In this paper, we slightly modify this definition and the section is taken from $H^0(X,E(\mathfrak{g}) \otimes K_X(D))$.
\end{rem}

\begin{exmp}\label{exmp para Higgs}
	As an example, let $G={\rm SL}_2(\mathbb{C})$, and we take $\theta={\rm diag}(\frac{1}{2},-\frac{1}{2})$. Then, the group scheme $\mathcal{G}_\theta$ over $\mathbb{D}$ can be regarded as elements of the form
	\begin{align*}
	\begin{pmatrix}
	A & z^{-1}A \\
	zA & A
	\end{pmatrix},
	\end{align*}
	with determinant one. Then, a local expression of a logahoric Higgs field $\varphi$ can be written in the form $\begin{pmatrix}
	A & z^{-1}A \\
	zA & A
	\end{pmatrix} \dfrac{dz}{z}$ with trace zero.
\end{exmp}

Denote by $\mathcal{M}_H(X,\mathcal{G}_{\boldsymbol\theta})$ the set of logahoric $\mathcal{G}_{\boldsymbol\theta}$-Higgs torsors on $X$, which also has a natural stack structure. Furthermore, we have a natural forgetful morphism of stacks
\begin{align*}
\mathcal{M}_H(X,\mathcal{G}_{\boldsymbol\theta}) \rightarrow {\rm Bun}(X,\mathcal{G}_{\boldsymbol\theta}),
\end{align*}
of which the fiber over a point $E \in {\rm Bun}(X,\mathcal{G}_{\boldsymbol\theta})$ is a finite module $H^0(X,E(\mathfrak{g}) \otimes K_X(D))$. Therefore, the forgetful morphism is representable and of finite type. The above discussion provides the following lemma, with an adaptation of Yun's proof for the same result (see \cite[Lemma 4.3.5]{Yun}):
\begin{lem}\label{302}
	The stack $\mathcal{M}_H(X,\mathcal{G}_{\boldsymbol\theta})$ is an algebraic stack locally of finite type.
\end{lem}

\subsection{Equivariant Logarithmic Higgs Bundles}
Let $Y$ be a smooth algebraic curve. Denote by $K_Y$ the canonical line bundle on $Y$ and let $\Gamma$ be a finite group together with an action on $Y$. Denote by $R$ a set of points of $Y$ (also a divisor), such that the stabilizer group $\Gamma_y$ is nontrivial for any $y \in R$. Let $G$ be a simply connected and semisimple linear algebraic group.

\begin{defn}\label{303}
	Given a set of representations ${\boldsymbol\rho}=\{\rho_y: \Gamma_y \rightarrow T, y \in R\}$, a \emph{$\Gamma$-equivariant logarithmic $G$-Higgs bundle} on $Y$ is a pair $(F,\phi)$, where
	\begin{itemize}
		\item $F$ is a $(\Gamma,G)$-bundle of type $\boldsymbol\rho$ on $Y$;
		\item $\phi$ is a $\Gamma$-equivariant element in $H^0(Y, F(\mathfrak{g}) \otimes K_Y(R))$.
	\end{itemize}
	The section $\phi$ is an \emph{(equivariant) logarithmic Higgs field}. For simplicity, this pair $(F,\phi)$ is called a \emph{logarithmic $(\Gamma,G)$-Higgs bundle}.
\end{defn}
Denote by $\mathcal{M}^{\boldsymbol\rho}_H(Y,\Gamma,G)$ the stack of logarithmic $(\Gamma,G)$-Higgs bundles of type $\boldsymbol\rho$ on $Y$. As discussed in \S 2.2, the data $(Y,R,\Gamma_y)$ uniquely determines a root stack $\mathcal{X}$. There is a canonical isomorphism of stacks
\begin{align*}
\mathcal{M}^{\boldsymbol\rho}_H(Y,\Gamma,G) \cong \mathcal{M}^{\boldsymbol\rho}_H(\mathcal{X},G),
\end{align*}
where $\mathcal{M}^{\boldsymbol\rho}_H(\mathcal{X},G)$ is the stack of $G$-bundles on $\mathcal{X}$. Note that $\mathcal{X}$ is also a Deligne--Mumford stack. As an application of Artin's theorem \cite{Art1974}, $\mathcal{M}^{\boldsymbol\rho}_H(\mathcal{X},G)$ is an algebraic stack locally of finite type \cite[Theorem 5.1]{Sun202003}. Then, we have the following lemma.

\begin{lem}\label{304}
	The stack $\mathcal{M}^{\boldsymbol\rho}_H(Y,\Gamma,G)$ is an algebraic stack locally of finite type.
\end{lem}

\subsection{Correspondence}
We first work on a formal disc $\mathbb{D}_y=\Spec(B)$ around a point $y \in R \subseteq Y$, and we will use the same notation as in \S 2.3. Let $\Gamma$ be a cyclic group of order $d$, and we have a natural $\Gamma$-action on $B=\mathbb{C}[[w]]$. Given  a weight $\theta$, denote by $\rho: \Gamma \rightarrow T$ the corresponding representation. Let $\Delta \in T(B)$ be the element such that $\Delta(\gamma w)=\rho(\gamma) \Delta(w)$.

Let $F$ be a $(\Gamma,G)$-bundle of type $\rho$ on $B$ and denote by $F(\mathfrak{g})$ the adjoint bundle. Without loss of generality, suppose that $F=G \times \mathbb{D}_y$. Let $\phi$ be an element in $\mathfrak{g}(\mathbb{C}[[w]]) \cdot \frac{dw}{w}$, which can be considered as a section $H^0(\mathbb{D}_y, F(\mathfrak{g}) \otimes \Omega^1_{\mathbb{D}}(y))$. Assume that $\phi$ is $\Gamma$-equivariant, that is,
\begin{align*}
\phi(\gamma w)=\rho(\gamma) \phi(w) \rho^{-1}(\gamma).
\end{align*}
Now consider $\varphi:=\Delta^{-1} \phi \Delta$ by conjugating with the matrix $\Delta$. Clearly, $\varphi$ is $\Gamma$-invariant:
\begin{align*}
\varphi(\gamma w)=\varphi(w).
\end{align*}
Therefore, $\varphi(w)$ descends to a section $\mathbb{D}_x \rightarrow E(\mathfrak{g}) \otimes \Omega^1_{\mathbb{D}_x}$ by substituting $z=w^d$, where $E$ is the $\mathcal{G}_\theta$-torsor corresponding to $F$. For each root $r \in R$, we have
\begin{align*}
\varphi(w)_r=\phi(w)_r w^{-d \cdot r(\theta)},
\end{align*}
and then, taking $z=w^d$, we get
\begin{align*}
\varphi(z)_r=\phi(z)_r z^{-r(\theta)}.
\end{align*}
Similar to the discussion in \S 2.3, the order of the pole of $\varphi(z)_r$ is bounded by $\lceil -r(\theta) \rceil$. Since $\varphi(z)_r$ is $\Gamma$-invariant, we have $\varphi(z)_r \in \mathfrak{u}_r(z^{m_r(\theta)}\mathbb{C}[[z]]) \otimes \frac{dz}{z}$ for each $r \in R$. In conclusion, the element $\varphi(z)$ is a section of $E(\mathfrak{g}) \otimes \Omega^1_{\mathbb{D}_x}(x)$ on $\mathbb{D}_x$. The above, in fact, describes a one-to-one correspondence, thus
\begin{align*}
H^0 (\mathbb{D}_y, F(\mathfrak{g}) \otimes  \Omega^1_{\mathbb{D}_y}(y))^{\Gamma} \cong H^0 (\mathbb{D}_x, E(\mathfrak{g}) \otimes \Omega^1_{\mathbb{D}_x}(x)).
\end{align*}

We next consider the correspondence globally. The setup is still the same as in \S 2.3. Let $X$ be a smooth algebraic curve with a fixed reduced effective divisor $D$. We fix a collection of rational weights $\boldsymbol\theta=\{\theta_x, x \in D\}$ and denote by $d_x$ the denominator of the rational weight $\theta_x$. Denote by $\pi: Y \rightarrow X$ the Galois covering determined by the above data. Let also $R\subseteq Y$ be the collection of pre-images of the points in $D$, which is the ramification divisor, and let $\boldsymbol\rho=\{\rho_y, y \in R\}$ be the set of representations corresponding to $\boldsymbol\theta$. We then have the following

\begin{thm}\label{305}
	With the notation above, there is an isomorphism
	\begin{align*}
	\mathcal{M}^{\boldsymbol\rho}_H(Y,\Gamma,G) \cong \mathcal{M}_H(X,\mathcal{G}_{\boldsymbol\theta})
	\end{align*}
	as algebraic stacks.
\end{thm}

\begin{proof}
	By Theorem \ref{206}, the stack of $(\Gamma,G)$-bundles on $Y$ is isomorphic to the stack of parahoric $\mathcal{G}_{\boldsymbol\theta}$-torsors on $X$,
	\begin{align*}
	{\rm Bun}^{\boldsymbol\rho}(Y,\Gamma,G) \cong {\rm Bun}(X,\mathcal{G}_{\boldsymbol\theta}).
	\end{align*}
	There are two natural forgetful morphisms of stacks
	\begin{align*}
	\mathcal{M}_H(X,\mathcal{G}_{\boldsymbol\theta}) \rightarrow {\rm Bun}(X,\mathcal{G}_{\boldsymbol\theta}) \quad \text{and} \quad \mathcal{M}^{\boldsymbol\rho}_H(Y,\Gamma,G) \rightarrow {\rm Bun}^{\boldsymbol\rho}(Y,\Gamma,G),
	\end{align*}
	with fibers $H^0(X, E(\mathfrak{g}) \otimes K_X(D))$ and $H^0(Y, F(\mathfrak{g}) \otimes K_Y(R))^{\Gamma}$, for $E \in {\rm Bun}(X,\mathcal{G}_{\boldsymbol\theta})$ and $F \in {\rm Bun}^{\boldsymbol\rho}(Y,\Gamma,G)$ the corresponding $(\Gamma,G)$-bundles. Therefore, proving that  $\mathcal{M}^{\boldsymbol\rho}_H(Y,\Gamma,G) \cong \mathcal{M}_H(X,\mathcal{G}_{\boldsymbol\theta})$ is equivalent to showing that
	\begin{align*}
	H^0(X, E(\mathfrak{g}) \otimes K_X(D)) \cong H^0(Y, F(\mathfrak{g}) \otimes K_Y(R))^{\Gamma},
	\end{align*}
	which is already proven at the beginning of this subsection.
\end{proof}

\section{Stability Conditions}\label{sect_stab_cond}

In this section, we study stability conditions of logahoric $\mathcal{G}_{\boldsymbol\theta}$-Higgs torsors and logarithmic $(\Gamma,G)$-Higgs bundles. We prove that these stability conditions are equivalent. The equivalence of the stability conditions helps us construct the moduli space of logahoric $\mathcal{G}_{\boldsymbol\theta}$-Higgs torsors in \S\ref{sect_moduli_space}. The stability conditions we study is inspired by the works of
Ramanathan (\cite[Lemma 2.1]{Rama1975} as well as \cite{Rama19961,Rama19962}) and is a generalization of Balaji and Seshadri's work \cite{BS}.

\subsection{Logahoric Higgs Torsors}\label{subsect_stab_parah}
Let $G$ be a connected complex reductive group. Let $\theta \in Y(T) \otimes_{\mathbb{Z}} \mathbb{Q}$ be a rational weight, and denote by $G_\theta \subseteq G(K)$ the parahoric group corresponding to $\theta$. Recall that a parabolic subgroup $P$ of $G$ can be determined by a subset of roots $R_P \subseteq R$. We define the following parahoric group as a subgroup of $P(K)$
\begin{align*}
P_\theta:=\langle T(A), U_r(z^{m_r(\theta)}A), r \in R_P \rangle.
\end{align*}
Denote by $\mathcal{P}_\theta$ the corresponding group scheme over $\mathbb{D}=\text{Spec}(A)$.

Now we consider the global picture. Let $X$ be a smooth algebraic curve with reduced effective divisor $D$. Let $\boldsymbol\theta=\{\theta_x, x\in D\}$ be a collection of rational weights and define the group scheme $\mathcal{P}_{\boldsymbol\theta}$ on $X$ by gluing the local data
\begin{align*}
\mathcal{P}_{\boldsymbol\theta}|_{\mathbb{D}_x} \cong P \times X\backslash D, \quad \mathcal{P}_{\boldsymbol\theta}|_{\mathbb{D}_x} \cong \mathcal{P}_{\theta_x}, x \in D.
\end{align*}
By \cite[Lemma 3.18]{ChGP}, the group scheme $\mathcal{P}_{\boldsymbol\theta}$ is a smooth affine group scheme of finite type, flat over $X$ and $\mathcal{P}_{\boldsymbol\theta} \subseteq \mathcal{G}_{\boldsymbol\theta}$.

Let $\kappa: \mathcal{P}_{\boldsymbol\theta} \rightarrow \mathbb{G}_m$ be a morphism of group schemes over $X$, which we call \emph{a character of $\mathcal{P}_{\boldsymbol\theta}$}. The following lemma provides an alternative way to understand characters of group schemes.
\begin{lem}\label{402}
	With the notation above, there is an isomorphism
	\begin{align*}
	{\rm Hom}(\mathcal{P}_{\boldsymbol\theta}, \mathbb{G}_m) \cong {\rm Hom}(P,\mathbb{C}^*)
	\end{align*}
	as sets.
\end{lem}

\begin{proof}
	Let $\kappa \in {\rm Hom}(\mathcal{P}_{\boldsymbol\theta}, \mathbb{G}_m)$ be a character, which is a morphism of schemes over $X$. Note that
	\begin{align*}
	\mathcal{P}_{\boldsymbol\theta}|_{X \backslash D} \cong P \times X \backslash D, \quad \mathbb{G}_m|_{X \backslash D} \cong \mathbb{C}^* \times X \backslash D.
	\end{align*}
	Therefore, restricting the morphism $\kappa$ to $X \backslash D$, we have
	\begin{align*}
	\kappa|_{X\backslash D} : \mathcal{P}_{\boldsymbol\theta}|_{X \backslash D} \rightarrow \mathbb{G}_m|_{X \backslash D}.
	\end{align*}
	Since the character $\kappa$ is a morphism of schemes over $X$, the restriction $\kappa|_{X\backslash D}$ uniquely determines a character $\chi: P \rightarrow \mathbb{C}^*$.
	
	Now we will consider the opposite direction and show that a character $\chi: P \rightarrow \mathbb{C}^*$ will uniquely determine a morphism $\kappa: \mathcal{P}_{\boldsymbol\theta} \rightarrow \mathbb{G}_m$. Note that $\mathcal{P}_{\boldsymbol\theta}|_{\mathbb{D}_x} \cong P \times X\backslash D$. A character $\chi$ gives a morphism $\kappa|_{X \backslash D}: \mathcal{P}_{\boldsymbol\theta}|_{\mathbb{D}_x} \rightarrow \mathbb{G}_m|_{X \backslash D}$. As a morphism of sheaves, we have the following commutative diagram
	\begin{center}
		\begin{tikzcd}
		\mathcal{P}_{\boldsymbol\theta}(X \backslash D \cap \mathbb{D}_x) \arrow[r, "\kappa|_{X \backslash D}"] \arrow[d, "\psi_x"] & \mathbb{G}_m(X \backslash D \cap \mathbb{D}_x) \arrow[d, "id"]\\
		\mathcal{P}_{\boldsymbol\theta}(\mathbb{D}_x \cap X \backslash D) \arrow[r, "\kappa|_{\mathbb{D}_x}"]  & \mathbb{G}_m(\mathbb{D}_x \cap X \backslash D),
		\end{tikzcd}
	\end{center}
	where $\psi_x$ is the transition function defining the group scheme $\mathcal{P}_{\boldsymbol\theta}$. Since $\mathbb{G}_m$ is a constant group scheme over $X$, its transition function is trivial. Since $X$ is connected, the commutativity of the above diagram will uniquely determine the morphism $\kappa|_{\mathbb{D}_x}$. Therefore, a character $\chi \in {\rm Hom}(P,\mathbb{C}^*)$ will uniquely determine a morphism $\kappa \in {\rm Hom}(\mathcal{P}_{\boldsymbol\theta}, \mathbb{G}_m)$.
\end{proof}

Given this lemma, whenever there is no ambiguity we shall be using the same notation $\chi$ for characters in ${\rm Hom}(\mathcal{P}_{\boldsymbol\theta}, \mathbb{G}_m)$ and ${\rm Hom}(P,\mathbb{C}^*)$. A character of $\mathcal{P}_{\boldsymbol\theta}$ will be called an \emph{anti-dominant character}, if the corresponding character $P \rightarrow \mathbb{C}^*$ is anti-dominant. Now let $E$ be a parahoric $\mathcal{G}_{\boldsymbol\theta}$-torsor on $X$. Let $\varsigma: X \rightarrow E/\mathcal{P}_{\boldsymbol\theta}$ be a reduction of structure group. Denote by $E_{\varsigma}$ the pullback of the following diagram
\begin{center}
	\begin{tikzcd}
	E_{\varsigma} \arrow[r, dotted] \arrow[d, dotted] & E \arrow[d]\\
	X \arrow[r, "\varsigma"]  & E/\mathcal{P}_{\boldsymbol\theta}.
	\end{tikzcd}
\end{center}
Then, the pullback $E_{\varsigma}$ is a $\mathcal{P}_{\boldsymbol\theta}$-torsor on $X$. Let $\chi: \mathcal{P}_{\boldsymbol\theta} \rightarrow \mathbb{G}_m$ be a character. We obtain a line bundle $\chi_* (E_{\varsigma})$ on $X$, and denote it by $L(\varsigma,\chi)$.

\begin{defn}\label{401}
	We define the \emph{parahoric degree} of a parahoric $\mathcal{G}_{\boldsymbol\theta}$-torsor $E$ with respect to a given reduction $\varsigma$ and a character $\chi$ as
	\begin{align*}
	parh\deg E(\varsigma,\chi)=\deg L(\varsigma,\chi) + \langle  \boldsymbol\theta, \chi \rangle,
	\end{align*}
	where $\langle \boldsymbol\theta, \chi \rangle:=\sum_{x\in D} \langle  \theta_x, \chi \rangle$.
\end{defn}

We introduce the following notion of stability for parahoric $\mathcal{G}_{\boldsymbol\theta}$-torsors inspired by the works of Ramanathan (\cite[Lemma 2.1]{Rama1975} as well as \cite{Rama19961,Rama19962}) on the construction of moduli spaces of semistable $G$-bundles on a projective nonsingular irreducible complex curve.

\begin{defn}\label{403}
	A parahoric $\mathcal{G}_{\boldsymbol\theta}$-torsor $E$ is called \emph{$R$-stable} (resp. \emph{$R$-semistable}), if for
	\begin{itemize}
		\item any proper parabolic group $P \subseteq G$,
		\item any reduction of structure group $\varsigma: X \rightarrow E/\mathcal{P}_{\boldsymbol\theta}$,
		\item any nontrivial anti-dominant character $\chi: \mathcal{P}_{\boldsymbol\theta} \rightarrow \mathbb{G}_m$, which is trivial on the center of $\mathcal{P}_{\boldsymbol\theta}$,
	\end{itemize}
	one has
	\begin{align*}
	parh\deg E(\varsigma,\chi) > 0, \quad (\text{resp. } \geq 0).
	\end{align*}
\end{defn}

\begin{rem}\label{Heinloth_stability}	When $G$ is semisimple, Definition \ref{403} is equivalent to \cite[Definition 6.3.4]{BS}. Note that the latter requires the degree smaller than zero and our Definition \ref{403} requires the degree bigger than zero. This difference comes from the setup of the construction, namely, we take anti-dominant characters while Balaji and Seshadri took dominant characters.  The second difference is that we consider reductive groups while the authors in \cite{BS} work with a semisimple and simply connected algebraic group. Thus, the condition that characters act trivially on the center always holds in Balaji and Seshadri's situation. We study this condition further in \S\ref{sect_R_mu_stab} and introduce the  $R_\mu$-stability condition. 
		
When we consider trivial parahoric weights, the stability condition for a parahoric $\mathcal{G}_{\boldsymbol\theta}$-torsor (that is, the case of $G$-bundles) is the same as the stability criterion in Corollary 1.16 of Heinloth \cite{He2}. However, when the weights are non-trivial, we have an extra term $\langle \boldsymbol\theta, \chi \rangle$ for weights in the definition of parahoric degree (see Definition \ref{401}). Thus, for non-trivial weights, our definition of ``degree" is different than that considered in \cite{He2}. 
\end{rem}

Now we move to the stability condition for logahoric $\mathcal{G}_{\boldsymbol\theta}$-Higgs torsors on $X$. Let $(E,\varphi)$ be a logahoric $\mathcal{G}_{\boldsymbol\theta}$-Higgs torsor on $X$, where $\varphi \in H^0(X,E(\mathfrak{g})\otimes K_X(D))$ is a section. A reduction of structure group $\varsigma: X \rightarrow E/\mathcal{P}_{\boldsymbol \theta}$ is said to be \emph{compatible with the logarithmic Higgs field} $\varphi$, if there is a lifting $\varphi':X \rightarrow E_{\varsigma}(\mathfrak{p}) \otimes K_X(D)$, such that the following diagram commutes
\begin{center}
	\begin{tikzcd}
	& E_{\varsigma}(\mathfrak{p}) \otimes K_X(D) \arrow[d, hook] \\
	X \arrow[r, "\varphi"] \arrow[ur, "\varphi'", dotted] & E(\mathfrak{g}) \otimes K_X(D).
	\end{tikzcd}
\end{center}

\begin{defn}\label{405}
	A logahoric $\mathcal{G}_{\boldsymbol\theta}$-Higgs torsor $(E,\varphi)$ is called \emph{$R$-stable} (resp. \emph{$R$-semistable}), if for
	\begin{itemize}
		\item any proper parabolic group $P \subseteq G$,
		\item any reduction of structure group $\varsigma: X \rightarrow E/\mathcal{P}_{\boldsymbol\theta}$ compatible with $\varphi$,
		\item any nontrivial anti-dominant character $\chi: \mathcal{P}_{\boldsymbol\theta} \rightarrow \mathbb{G}_m$, which is trivial on the center of $\mathcal{P}_{\boldsymbol\theta}$,
	\end{itemize}
	one has
	\begin{align*}
	parh\deg E(\varsigma,\chi) > 0, \quad (\text{resp. } \geq 0).
	\end{align*}
\end{defn}

\begin{rem}\label{406}
	Recall that a classical non-parabolic Higgs bundle $(E,\varphi)$ on $X$ is stable, if for any $\varphi$-invariant subbundle $F \subseteq E$, one has $\frac{\deg F}{ {\rm rk} F} < \frac{\deg E}{ {\rm rk} E}$. For $(E,\varphi)$ a logahoric $\mathcal{G}_{\boldsymbol\theta}$-Higgs torsor, a reduction of structure group $\varsigma:X \rightarrow E/\mathcal{P}_{\boldsymbol\theta}$ compatible with the logarithmic Higgs field $\varphi$ is actually giving a ``$\varphi$-invariant subbundle". 
	
	We use Example \ref{exmp para Higgs} to explain this idea locally. Let $\varphi=\begin{pmatrix}
	1 & 0 \\
	z & -1
	\end{pmatrix} \dfrac{dz}{z}$ be a logahoric $(\mathcal{SL}_2)_{\theta}$-Higgs field. Denote by $\alpha$ the unique positive root of ${\rm SL}_2$. For $(\mathcal{SL}_2)_{\theta}$, there are two subgroup schemes, which are defined by $\alpha$ and $-\alpha$ respectively, namely,
	\begin{align*}
	\mathcal{P}_{\alpha}=\begin{pmatrix}
	A & z^{-1}A \\
	0 & A
	\end{pmatrix}, \quad \mathcal{P}_{-\alpha}=\begin{pmatrix}
	A & 0 \\
	zA & A
	\end{pmatrix}.
	\end{align*}
	It is easy to check that the logarithmic Higgs field $\varphi$ cannot be lifted to a section of $\mathfrak{p}_{\alpha} \otimes \frac{dz}{z}$, where $\mathfrak{p}_\alpha$ is the ``Lie algebra" of $\mathcal{P}_\alpha$ (see \cite{Bo}). On the other hand, it can be naturally lifted to a section of $\mathfrak{p}_{-\alpha} \otimes \frac{dz}{z}$. When we consider the associated sheaf, this property can be also understood as the logarithmic Higgs field $\varphi$ preserving the subsheaf $\begin{pmatrix} 0 \\ \ast \end{pmatrix}$, but not preserving $\begin{pmatrix} \ast \\ 0 \end{pmatrix}$.
\end{rem}

\subsection{Equivariant Logarithmic Higgs Bundles}
Let $Y$ be a smooth algebraic curve with a $\Gamma$-action. Denote by $R \subseteq Y$ the set of points such that the stabilizer group $\Gamma_y$ is nontrivial for each $y \in R$. Note that there is a correspondence between weights $\boldsymbol\theta$ and type $\boldsymbol\rho$. For simplicity, we say a $(\Gamma,G)$-bundle is of type $\boldsymbol\theta$, if it comes from a parahoric $\mathcal{G}_{\boldsymbol\theta}$-torsor.

Let $F$ be a $(\Gamma,G)$-bundle of type $\boldsymbol\theta$ on $Y$. Let $P \subseteq G$ be a parabolic subgroup. Given a $\Gamma$-equivariant reduction $\sigma: Y \rightarrow F/P$, the pullback 
\begin{center}
	\begin{tikzcd}
	F_\sigma \arrow[r, dotted] \arrow[d, dotted] & F \arrow[d]\\
	Y \arrow[r, "\sigma"]  & F/P
	\end{tikzcd}
\end{center}
is a $(\Gamma,P)$-bundle on $Y$. Let $\chi: P \rightarrow \mathbb{C}^*$ be a character. We obtain a line bundle $L(\sigma,\chi):=\chi_* F_\sigma$ on $Y$.

\begin{defn}\label{407}
	We define the \emph{degree} of a $(\Gamma,G)$-bundle $F$ with respect to a given reduction $\sigma$ and a character $\chi$ as
	\begin{align*}
	\deg F(\sigma,\chi)=\deg L(\sigma,\chi).
	\end{align*}
\end{defn}

\begin{defn}\label{408}
	A $(\Gamma,G)$-bundle $F$ of type $\boldsymbol\theta$ is called \emph{R-stable} (resp. \emph{R-semistable}), if for
	\begin{itemize}
		\item any proper parabolic group $P \subseteq G$,
		\item any $\Gamma$-equivariant reduction of structure group $\sigma: Y \rightarrow F/P$,
		\item any nontrivial anti-dominant character $\chi: P \rightarrow \mathbb{C}^*$, which is trivial on the center of $P$,
	\end{itemize}
	one has
	\begin{align*}
	\deg F(\sigma,\chi) >0, \quad (\text{resp. } \geq 0).
	\end{align*}
\end{defn}

Let $(F,\phi)$ be a logarithmic $(\Gamma,G)$-Higgs bundle on $Y$. Given a reduction $\sigma: Y \rightarrow F/P$, we have a natural morphism $F_\sigma \rightarrow F$ given by the pullback. This induces a natural morphism
\begin{align*}
F_\sigma(\mathfrak{p}) = F_\sigma \times_P \mathfrak{p} \rightarrow F_\sigma \times_G  \mathfrak{g} \rightarrow F \times_G \mathfrak{g}=F(\mathfrak{g}).
\end{align*}
A reduction of structure group $\sigma: Y \rightarrow F/P$ is \emph{compatible with} the logarithmic Higgs field $\phi$, if there is a lifting $\phi':Y \rightarrow F_\sigma(\mathfrak{p}) \otimes K_Y(R)$ such that the following diagram commutes
\begin{center}
	\begin{tikzcd}
	& F_\sigma(\mathfrak{p}) \otimes K_Y(R) \arrow[d, hook] \\
	Y \arrow[r, "\phi"] \arrow[ur, "\phi'", dotted] & F(\mathfrak{g}) \otimes K_Y(R).
	\end{tikzcd}
\end{center}

\begin{defn}\label{409}
	A logarithmic $\left( \Gamma, G \right)$-Higgs bundle $(F,\phi)$ on $Y$ is called \emph{R-stable} (resp. \emph{R-semistable}), if for
	\begin{itemize}
		\item any proper parabolic subgroup $P$ of $G$,
		\item any $\Gamma$-equivariant reduction of structure group $\sigma: Y \rightarrow F/P$ compatible with $\phi$,
		\item any nontrivial anti-dominant character $\chi: P \rightarrow \mathbb{C}^*$, which is trivial on the center of $P$,
	\end{itemize}
	one has
	\begin{align*}
	\deg F(\sigma,\chi) >0, \quad (\text{resp. } \geq 0).
	\end{align*}
\end{defn}




\subsection{Equivalence of Stability Conditions}
In this subsection, we will prove that a logahoric $\mathcal{G}_{\boldsymbol\theta}$-Higgs torsor $(E,\varphi)$ on $X$ is $R$-stable (resp. $R$-semistable) if and only if the corresponding logarithmic $(\Gamma,G)$-Higgs bundle $(F,\phi)$ on $Y$ is $R$-stable (resp. $R$-semistable). By Definition \ref{403} and Definition \ref{409}, we have to show the following correspondences
\begin{enumerate}
	\item every parabolic subgroup $P \subseteq G$ corresponds to a subgroup scheme $\mathcal{P}_{\boldsymbol\theta} \subseteq \mathcal{G}_{\boldsymbol\theta}$;
	\item there is a one-to-one correspondence between characters ${\rm Hom}(P, \mathbb{C}^*)$ and ${\rm Hom}(\mathcal{P}_{\boldsymbol\theta}, \mathbb{G}_m)$;
	\item every reduction of structure group $\varsigma: X \rightarrow E/\mathcal{P}_{\boldsymbol\theta}$ corresponds to a unique $\Gamma$-equivariant reduction of structure group $\sigma: Y \rightarrow F/P$;
	\item $\deg F(\sigma,\chi) \geq 0$ (resp. $>$) if and only if $parh\deg E(\varsigma,\chi) \geq 0$ (resp. $>$).
\end{enumerate}
The first condition holds from the Definitions in \S\ref{subsect_stab_parah}. Lemma \ref{402} gives us the correspondence between characters. The third and fourth conditions will be proved in Lemmas \ref{413} and \ref{414} below.

We review first the construction of a $(\Gamma,G)$-bundle $F$ by gluing local data from \S\ref{subsect_equi_bundle}. For each $y \in R$, we define $F_y:= \mathbb{D}_y  \times G$ and define $F_0:=(Y \backslash R) \times G$ together with the $\Gamma$-actions
\begin{align*}
& \gamma \cdot (u,g) \rightarrow (\gamma u, \rho_y(\gamma)g),\quad  u \in \mathbb{D}_y, \gamma \in \Gamma_y,\\
& \gamma \cdot (u,g) \rightarrow (\gamma u, g), \quad \quad  u \in Y \backslash R, \gamma \in \Gamma_y.
\end{align*}
By giving $(\Gamma,G)$-isomorphisms
\begin{align*}
\Theta_y: F_y|_{\mathbb{D}^{\times}_y} \rightarrow F_0|_{\mathbb{D}^{\times}_y}, \quad y \in R,
\end{align*}
we can define a $(\Gamma,G)$-bundle $F$ of type $\boldsymbol\rho$ on $Y$. Note that a $(\Gamma,G)$-isomorphism $\Theta_y$ satisfies the following condition
\begin{align*}
\Theta_y (\gamma w)=\rho(\gamma) \Theta_y(w).
\end{align*}
A $(\Gamma,G)$-bundle is usually not $\Gamma$-invariant. However, there is a canonical way to construct a $\Gamma$-invariant $G$-bundle $F'$ based on $F$. Here is the construction. On each punctured disc $\mathbb{D}^{\times}_y$, we define a new $(\Gamma,G)$-isomorphism
\begin{align*}
\Theta'_y(w):=\Delta(w)^{-1} \Theta_y(w),
\end{align*}
where $\Delta$ is the element in $T(\mathbb{D}_y^{\times})$ such that $\Delta(\gamma w)=\rho(\gamma)\Delta(w)$ (see \S\ref{subsect_corr_torsor_and_bundle}). Denote by $F'$ the $(\Gamma,G)$-bundle given by the isomorphisms $\Theta'_y$. It is easy to check that
\begin{align*}
\Theta'_y(\gamma w)=\Theta'_y(w),
\end{align*}
which implies that $F'$ is $\Gamma$-invariant. Actually, the correspondence we reviewed in \S\ref{subsect_corr_torsor_and_bundle} is given by taking the invariance of $F'$.

\begin{lem}\label{413}
	Let $E$ be a parahoric $\mathcal{G}_{\boldsymbol\theta}$-torsor on $X$. Denote by $F$ the corresponding $(\Gamma,G)$-bundle on $Y$. Then, we have
	\begin{align*}
	{\rm Hom}(X, E/\mathcal{P}_{\boldsymbol\theta}) \cong {\rm Hom}^{\Gamma}(Y, F/G),
	\end{align*}
	which describes a one-to-one correspondence between reductions of structure group of $E$ and $\Gamma$-equivariant reductions of structure group of $F$.
\end{lem}

\begin{proof}
	Given a $\Gamma$-equivariant reduction $\sigma: Y \rightarrow F/P$, we have
	\begin{align*}
	\sigma(\gamma w)= \rho(\gamma) \sigma(w) \rho(\gamma)^{-1},
	\end{align*}
	where $\gamma \in \Gamma$. A $\Gamma$-equivariant reduction uniquely determines a $\Gamma$-equivariant reduction
	\begin{align*}
	\sigma': Y \rightarrow F'/P
	\end{align*}
	of $F'$, which is the $\Gamma$-invariant $G$-bundle constructed from $F$. Since $F'$ is $\Gamma$-invariant, we have
	\begin{align*}
	\sigma'(\gamma u)=\sigma'(u).
	\end{align*}
	By taking invariants under $\Gamma$, we get a section $\varsigma: X \rightarrow E/\mathcal{P}_{\boldsymbol\theta}$. This process also holds in the other direction.
\end{proof}

\begin{lem}\label{414}
	Let $E$ be a parahoric $\mathcal{G}_{\boldsymbol\theta}$-torsor on $X$ and denote by $F$ the corresponding $(\Gamma,G)$-bundle of type ${\boldsymbol\theta}$ on $Y$. Let $d$ be the order of the group $\Gamma$. Let $\varsigma: X \rightarrow E/\mathcal{P}_{\boldsymbol\theta}$ be a reduction of structure group of $E$, and let $\sigma: Y \rightarrow F/P$ be the corresponding $\Gamma$-equivariant reduction of structure group of $F$. Let $\chi: P \rightarrow \mathbb{C}^*$ be a character. Then, the following identity holds
	\begin{align*}
	d \cdot parh\deg E(\varsigma,\chi)=\deg F(\sigma,\chi).
	\end{align*}
\end{lem}

\begin{proof}
	For simplicity, we assume that $D=\{x\}$ and $R=\{y\}$ are singletons; the proof when $D$ and $R$ have finitely many points is entirely analogous. The stabilizer group of $y \in R$ is $\Gamma$, which is a cyclic group of order $d$. Recalling the construction of $F'$, we define a new transition function $\Theta'_y(w)$ such that
	\begin{align*}
	\Theta'_y(w):=\Delta(w)^{-1} \Theta_y(w),
	\end{align*}
	where $\Theta_y: F_y|_{\mathbb{D}^{\times}_y} \rightarrow F_0|_{\mathbb{D}^{\times}_y}$ is the transition function of $F$ and $\Delta$ is the matrix in $T(\mathbb{D}_y)$ such that $\Delta(\gamma w)=\rho(\gamma)\Delta(w)$. Therefore, $\Delta$ contributes to the degree $\sum_{r \in R} d\cdot r(\theta_x)$. Therefore, we have
	\begin{align*}
	\deg F(\sigma,\chi)=\deg F'(\sigma',\chi)+ d \cdot  \langle \theta_x, \chi \rangle.
	\end{align*}
	Since $F'$ is a $\Gamma$-invariant $G$-bundle, we have
	\begin{align*}
	d \cdot \deg E(\varsigma,\chi)= \deg F'(\sigma',\chi).
	\end{align*}
	By adding the term $d \cdot  \langle \theta_x, \chi \rangle$ on both sides of the equation, we finally get
	\begin{align*}
	\deg F(\sigma,\chi) = d \cdot parh\deg E(\varsigma,\chi),
	\end{align*}
	which is the desired identity.
\end{proof}

\begin{thm}\label{415}
	Let $E$ be a parahoric $\mathcal{G}_{\boldsymbol\theta}$-torsor on $X$, and let $F$ be the corresponding $(\Gamma,G)$-bundle on $Y$. Then, $E$ is $R$-stable (resp. $R$-semistable) if and only if $F$ is $R$-stable (resp. $R$-semistable).
\end{thm}

\begin{proof}
	We will prove that if a $(\Gamma,G)$-bundle $F$ is $R$-stable, then the corresponding parahoric $\mathcal{G}_{\boldsymbol\theta}$-torsor $E$ is also $R$-stable. The other direction can be proved similarly. By Definition \ref{403}, we have to show that for every proper parabolic subgroup $P \subseteq G$, every nontrivial anti-dominant character $\chi: \mathcal{P}_{\boldsymbol\theta} \rightarrow \mathbb{G}_m$ and every reduction of structure group $\varsigma: X \rightarrow E/ \mathcal{P}_{\boldsymbol\theta}$, we have
	\begin{align*}
	parh\deg E(\varsigma,\chi) > 0.
	\end{align*}
	From Lemma \ref{402}, the character $\chi$ can be considered as an anti-dominant character $P \rightarrow \mathbb{C}^*$, while Lemma \ref{413} provides that the reduction of structure group $\varsigma$ corresponds to a $\Gamma$-equivariant reduction $\sigma: Y \rightarrow F/P$. Suppose that $F$ is $R$-stable, which means that
	\begin{align*}
	\deg F(\sigma,\chi)>0.
	\end{align*}
	By Lemma \ref{414}, we have
	\begin{align*}
	parh\deg E(\varsigma,\chi) = \frac{1}{d} \deg F(\sigma,\chi)>0,
	\end{align*}
	and so $E$ is $R$-stable.
\end{proof}

\begin{thm}\label{416}
	Let $(E,\varphi)$ be a logahoric $\mathcal{G}_{\boldsymbol\theta}$-Higgs torsor on $X$, and let $(F,\phi)$ be the corresponding logarithmic $(\Gamma,G)$-Higgs bundle on $Y$. Then, $(E,\varphi)$ is $R$-stable (resp. $R$-semistable) if and only if $(F,\phi)$ is $R$-stable (resp. $R$-semistable).
\end{thm}

\begin{proof}
	The only thing we have to show is that a reduction $\varsigma: X \rightarrow E/\mathcal{P}_{\boldsymbol\theta}$ is compatible with $\varphi$ if and only if the corresponding $\Gamma$-equivariant reduction $\sigma: Y \rightarrow F/P$ is compatible with $\phi$. We still prove one direction, that if $\sigma$ is compatible with $\phi$, then $\varsigma$ is compatible with $\varphi$. The other direction can be proved similarly.
	
	From the assumption, there is a lifting $\phi':Y \rightarrow F_\sigma(\mathfrak{p}) \otimes K_Y(R)$, such that the following diagram commutes
	\begin{center}
		\begin{tikzcd}
		& F_\sigma(\mathfrak{p}) \otimes K_Y(R) \arrow[d, hook] \\
		Y \arrow[r, "\phi"] \arrow[ur, "\phi'", dotted] & F(\mathfrak{g}) \otimes K_Y(R).
		\end{tikzcd}
	\end{center}
	Note that $F_\sigma$ is a $(\Gamma,P)$-bundle, for which the $\Gamma$-action is induced from that on $F$. Therefore, $\phi'$ is also $\Gamma$-equivariant, that is,
	\begin{align*}
	\phi'(\gamma w)=\rho(\gamma) \phi'(w) \rho(\gamma)^{-1}.
	\end{align*}
	By \S 3.3, the $\Gamma$-equivariant logarithmic Higgs field $\phi'$ will correspond to a logarithmic Higgs field $\varphi': X \rightarrow E_\varsigma(\mathfrak{p}) \otimes K_X(D)$, where $\varphi'$ is a lifting of $\varphi$. Clearly, this correspondence is a one-to-one correspondence.
\end{proof}


\section{$R_\mu$-stability Condition}\label{sect_R_mu_stab}

In this section, we first study logahoric ${\rm GL}_n$-Higgs torsors in detail as an important example. In \S 5.2, we introduce the $R_\mu$-stability condition for logahoric ${\rm GL}_n$-Higgs torsors. Compared to the $R$-stability condition from Definition \ref{403}, we do not require that the anti-dominant character acts trivially on the center (see Definition \ref{504}). We show that for a specific $\mu$, the notion of $R_\mu$-stability coincides with the stability condition for a parabolic Higgs bundle as considered by Simpson in \cite{Simp} (see Proposition \ref{505}). In \S 5.3, we generalize the $R_\mu$-stability condition to the case of arbitrary complex reductive groups, and prove that we can find a canonical $\mu$ such that $R$-stability is equivalent to $R_\mu$-stability (Proposition \ref{alphaexst}).

\subsection{Correspondence}
Let $T \subset {\rm GL}_n$ be the subgroup of diagonal matrices, which is a maximal torus in ${\rm GL}_n$. Let $\theta$ be a rational weight in $Y(T) \otimes_{\mathbb{Z}}\mathbb{Q}$. Regarding $\theta$ as an element in $\mathfrak{t}_{\mathbb{Q}}$, we have $\theta=\sum_{i=1}^n \frac{a_i}{d}t_i$, where $a_{i}$ and $d$ are integers and $\{t_i\}_{1 \leq i \leq n}$ is a basis of $\mathfrak{t}_{\mathbb{Q}}$. Equivalently, $\theta$ can be regarded as a diagonal matrix
\begin{align*}
\theta=
\begin{pmatrix}
\frac{a_1}{d} &  & \\
& \ddots & \\
& & \frac{a_n}{d}
\end{pmatrix}.
\end{align*}
Furthermore, we assume that $a_i \leq a_j$, if $i \leq j$. In the case when $\theta$ is small, one has that $0 \leq a_i < d$. Denote by ${\rm GL}_{\theta}$ the parahoric subgroup of ${\rm GL}_n(K)$ defined as
\begin{align*}
{\rm GL}_\theta:= \langle T(A), U_{ij}(z^{\lceil a_j-a_i \rceil} A) \rangle,
\end{align*}
where $U_{ij}$ is the unipotent group corresponding to the $(i,j)$-entry. As an example, for $n=2$ and $\theta=0 \cdot t_1+\frac{1}{2}t_2$, the matrix in ${\rm GL}_{\theta}$ can be written as
\begin{align*}
\begin{pmatrix}
A & A \\
zA & A
\end{pmatrix}.
\end{align*}

The corresponding representation $\rho:\Gamma \rightarrow T$ is given by
\begin{align*}
\rho(\gamma)=
\begin{pmatrix}
\xi^{a_1} & & \\
& \ddots &\\
& & \xi^{a_n}
\end{pmatrix},
\end{align*}
where $\xi=e^{\frac{2\pi i}{d}}$. In the local coordinate $w$, we define the following matrix
\begin{align*}
\Delta(w):=w^{\theta}=
\begin{pmatrix}
w^{a_1} & & \\
& \ddots &\\
& & w^{a_n}
\end{pmatrix}.
\end{align*}
Clearly, we have
\begin{align*}
\Delta(\gamma w)=\rho(\gamma) \Delta(w).
\end{align*}

Let $F$ be a $(\Gamma,{\rm GL}_n)$-bundle of type $\rho$ over $\mathbb{D}_y$, and denote by $\mathbb{U}:={\rm Aut}_{(\Gamma,{\rm GL}_n)}(F)$ the automorphisms of $F$. We take an element $\sigma \in \mathbb{U}$. Note that in the ${\rm GL}_n$-case, we can consider $\sigma=(\sigma_{ij})$ as a matrix. Define $\varsigma:=\Delta^{-1}\sigma\Delta$. Then, we have
\begin{align*}
\varsigma_{ij}(w)=\sigma_{ij}(w) w^{-(a_i-a_j)},
\end{align*}
and it is easy to check that
\begin{align*}
\varsigma(\gamma w)=\varsigma(w),
\end{align*}
which means that $\varsigma$ is $\Gamma$-invariant. Substituting $z$ by $w^d$, then $\varsigma_{ij}(z)$ descends to a meromorphic function on $\mathbb{D}_x= \Spec(A)$. The order of the pole of $\varsigma_{ij}(z)$ is bounded by $\lceil \frac{a_j-a_i}{d} \rceil$. Therefore, we have $\varsigma(z) \in {\rm GL}_\theta$. The same argument as for Theorem \ref{207}, provides the following corollary:
\begin{cor}\label{501}
	With respect to the above notation, there is an isomorphism
	\begin{align*}
	{\rm Bun}^{\boldsymbol\rho}(Y,\Gamma) \cong {\rm Bun}(X,{\rm GL}_{\boldsymbol\theta})
	\end{align*}
	as algebraic stacks.
\end{cor}

Now let $\phi=(\phi_{ij})$ be an element in $\mathfrak{gl}(\mathbb{C}[[w]]) \cdot \frac{dw}{w}$, where $\phi_{ij}(w)$ is a holomorphic function corresponding to the $(i,j)$-entry. The element $\phi$ can be considered as a logarithmic Higgs field over $\mathbb{D}_y$. Suppose that $\phi$ is $\Gamma$-equivariant, that is, $\phi(\gamma w)=\rho(\gamma) \phi(w) \rho^{-1}(\gamma)$. Define $\varphi:=\Delta^{-1} \phi \Delta$. Then, we have
\begin{align*}
\varphi(\gamma w)=\varphi(w),
\end{align*}
which implies that $\varphi$ is $\Gamma$-invariant. Therefore, $\varphi$ can be descended to a section $\mathbb{D}_x \rightarrow \mathfrak{gl}_{\theta} \otimes \Omega_{\mathbb{D}_x}^1(x)$ by substituting $z$ by $=w^d$, where $\mathfrak{gl}_{\theta}$ is the Lie algebra of ${\rm GL}_{\theta}$. For each entry, we have
\begin{align*}
\varphi(z)_{ij}= \phi(z)_{ij} z^{-(a_i-a_j)}.
\end{align*}

Globally, let $E$ be a parahoric ${\rm GL}_{\boldsymbol\theta}$-torsor on $X$, where $\boldsymbol\theta=\{\theta_x, x \in D\}$ is a collection of rational weights over points in the divisor $D$. Let $F$ be the corresponding $\Gamma$-equivariant bundle of type $\boldsymbol\rho=\{\rho_y, y \in R\}$ on $Y$, where $R$ is the pre-image of $D$. Then, the above discussion implies that there is a one-to-one correspondence between $\Gamma$-equivariant logarithmic Higgs fields of $F$ and logarithmic Higgs fields of $E$,
\begin{align*}
H^0(Y, \mathcal{E}nd(F) \otimes K_Y(R))^{\Gamma} \cong H^0(X, E(\mathfrak{g}) \otimes K_X(D)).
\end{align*}

As an application of Theorem \ref{305}, we have the following equivalence.
\begin{cor}\label{502}
	One has
	\begin{align*}
	\mathcal{M}_H^{\boldsymbol\rho}(Y,\Gamma) \cong \mathcal{M}_H(X,{\rm GL}_{\boldsymbol\theta})
	\end{align*}
	as algebraic stacks.
\end{cor}

\begin{rem}\label{503}
	Mehta and Seshadri introduced parabolic bundles to study $\Gamma$-equivariant bundles \cite{MS}. In fact, there is a correspondence between $\Gamma$-equivariant bundles on $Y$ and parabolic bundles on $X$ \cite{Bis1997,NaSt}. Furthermore, this equivalence was extended for strongly parabolic Higgs bundles with rank two \cite{NaSt}, and more generally in \cite{BMW}. For primary reference on parabolic Higgs bundles, we refer the reader to \cite{BoYo, GGM, Simp}; examples of parabolic $G$-Higgs bundles for complex groups $G$ as special parabolic Higgs bundle data are demonstrated in \cite{KSZ211,KSZ212}.
\end{rem}

\subsection{$R_\mu$-Stability Condition}
In this subsection, we define a new stability condition called the $R_\mu$-stability condition, and we prove that the $R_{\mu}$-stability condition
of a parahoric torsor in the case ${\rm GL}_n$ is equivalent to the stability condition of the corresponding parabolic bundle under the condition that all weights are small. Note that the condition that weights are small is not a strong condition because any parahoric subgroup ${\rm GL}_{\boldsymbol\theta}$ of ${\rm GL}_n$ is conjugate to some ${\rm GL}_{\boldsymbol\theta_0}$, for $\boldsymbol\theta_0$ a collection of small weights. 

Let $\boldsymbol\theta$ be a set of rational weights over the points in $D \subseteq X$. Let $V$ be a free vector bundle on $X$, that is, let $V \cong X \times \mathbb{C}^n$. Viewing $V$ as a sheaf, there is a natural ${\rm GL}_{\boldsymbol\theta}$-action on $V$. Let $P \subset {\rm GL}_n$ be a parabolic group. Denote by $\mathcal{P}_{\boldsymbol\theta}$ the corresponding parahoric group scheme. Let $\kappa \in {\rm Hom}(\mathcal{P}_{\boldsymbol\theta},\mathbb{G}_m)$ be a character. By Lemma \ref{402}, we know that it is equivalent to a character $\chi: P \rightarrow \mathbb{C}^*$. Let $E$ be a ${\rm GL}_{\boldsymbol\theta}$-torsor on $X$. Let $\varsigma: X \rightarrow E/ \mathcal{P}_{\boldsymbol\theta}$ be a reduction of structure group.

For the calculation below, we assume that the divisor $D=\{x\}$ consists of a single point and so we shall denote $\boldsymbol\theta=\theta_x$ for simplicity; the proof for finitely many points in $D$ is then analogous. There is a natural ${\rm GL}_n$-action on $\mathbb{C}^n$, which induces a $P$-action on this vector space. Note that the differentiation $d\chi$ is an element in ${\rm Hom}(\mathfrak{p},\mathbb{C}) \subseteq {\rm Hom}(\mathfrak{t},\mathbb{C})$, where $\mathfrak{p}$ is the Lie algebra of $P$. Denote by $s_{\chi} \in \mathfrak{t}$ the dual of $d\chi$, and let $\lambda_1, \dots, \lambda_r$ be the eigenvalues of $s_{\chi}$. We assume that $\lambda_1 < \lambda_2 < \dots < \lambda_r$. Define
\begin{align*}
V_j:={\rm ker}(\lambda_j I - s_{\chi})
\end{align*}
to be the eigenspace of $\lambda_j$. In fact, we consider $V_j$ as a trivial vector bundle on $X$. Then, we have
\begin{align*}
\deg L(\varsigma,\chi)=\sum_{j=1}^r \lambda_j \deg \mathcal{V}_j,
\end{align*}
where $\mathcal{V}_j:=E \times_{{\rm GL}_{\theta_x}} V_j$. In fact, the reduction $\varsigma$ and the character $\chi$ determine sub-torsors (or subbundles) of $E$, and all sub-torsors of $E$ can be constructed in this way.

Suppose that $\theta_x: \mathbb{C}^* \rightarrow T$ is given by
\begin{align*}
z \rightarrow
\begin{pmatrix}
z^{\alpha_1} & & \\
& \ddots &\\
& & z^{\alpha_n}
\end{pmatrix}.
\end{align*}
Let $\alpha_1,\dots,\alpha_l$ be all (distinct) eigenvalues of the differentiation $d \theta_x$ such that $\alpha_1 > \alpha_2 > \dots > \alpha_l$, and denote by $A_i$ the eigenspace of $\alpha_i$, $1 \leq i \leq l$. Then, we have
\begin{align*}
\langle \theta_x, \chi \rangle=\sum_{i,j} \alpha_i \lambda_j \dim(A_i \cap V_j).
\end{align*}
For $B_i:=A_i \oplus \dots \oplus A_1$, we have a natural filtration of $V_j$ given by
\begin{align*}
B_l \cap V_j \supseteq B_{l-1} \cap V_j \supseteq \dots \supseteq B_1 \cap V_j \supseteq \{0\},
\end{align*}
together with a collection of weights
\begin{align*}
\alpha_l \leq \alpha_{l-1} \leq \dots \leq \alpha_1.
\end{align*}
This defines a parabolic structure of $\mathcal{V}_j$ on the fiber of the point $x$.

Given the discussion above, we now see
\begin{equation*}\label{ast}\tag{$\ast$}
\begin{split}
parh \deg E(\varsigma,\chi)& =\sum_{j=1}^r \lambda_j \deg \mathcal{V}_j + \sum_{i,j} \alpha_i \lambda_j \dim(A_i \cap V_j)\\
& = \sum_{j=1}^r \lambda_j (\deg \mathcal{V}_j + \sum_{i=1}^l \alpha_i \dim(A_i \cap V_j))\\
& = \sum_{j=1}^r \lambda_j par \deg \mathcal{V}_j.
\end{split}
\end{equation*}
Therefore, the parahoric degree in the case of ${\rm GL}_n$ \textit{coincides with the parabolic degree} of the induced parabolic bundle.

We can find a unique rational weight $\varpi$ such that for any character $\chi$ defined as above, we have
\begin{align*}
\langle \varpi, \chi \rangle=\sum_{j=1}^r \lambda_j \dim(V_j).
\end{align*}

\begin{defn}\label{504}
	Given a rational number $\mu$, a parahoric ${\rm GL}_{\boldsymbol\theta}$-torsor $E$ is called \emph{$R_{\mu}$-stable} (resp. \emph{$R_{\mu}$-semistable}), if for
	\begin{itemize}
		\item any proper parabolic group $P \subseteq G$,
		\item any reduction of structure group $\varsigma: X \rightarrow E/\mathcal{P}_{\boldsymbol\theta}$,
		\item any nontrivial anti-dominant character $\chi: \mathcal{P}_{\boldsymbol\theta} \rightarrow \mathbb{G}_m$ (not necessarily trivial on the center $\mathfrak{z}$),
	\end{itemize}
	one has
	\begin{align*}
	parh \deg E(\varsigma,\chi)-\langle \mu \varpi, \chi \rangle > 0 \quad (\text{resp. } \geq 0).
	\end{align*}
\end{defn}

\noindent Note that this definition can be generalized to define stability in the case of any \textit{real reductive group}; we refer the reader to \cite{BGM} for more information.

Let $E$ be a parahoric ${\rm GL}_{\boldsymbol\theta}$-torsor. As was discussed in Remark \ref{503}, $E$ can be considered as a parabolic bundle, so we shall keep the same notation $E$ to refer to it. The following proposition relates the stability condition of a logahoric ${\rm GL}_{\boldsymbol\theta}$-Higgs torsor with the stability condition of Simpson from \cite{Simp} for parabolic Higgs bundles.

\begin{prop}\label{505}
	Let $\boldsymbol\theta$ be a collection of small weights. Let $E$ be a parahoric ${\rm GL}_{\boldsymbol\theta}$-torsor. Denote by $\mu:=\frac{par\deg E}{rk (E)}$ the parabolic slope of $E$ as a parabolic bundle. Then, $E$ is $R_\mu$-stable (resp. $R_{\mu}$-semistable) if and only if $E$ is stable (resp. semistable) as a parabolic bundle. Furthermore, let $(E,\varphi)$ be a logahoric ${\rm GL}_{\boldsymbol\theta}$-Higgs bundle. Then, $(E,\varphi)$ is $R_\mu$-stable (resp. $R_{\mu}$-semistable) if and only if $(E,\varphi)$ is stable (resp. semistable) as a parabolic Higgs bundle.
\end{prop}

\begin{proof}
	We only give the proof for parahoric ${\rm GL}_{\boldsymbol\theta}$-torsors, and the proof for the case of logahoric ${\rm GL}_{\boldsymbol\theta}$-Higgs torsors is, in fact, the same. We follow the same notation as was used in this subsection, and still work for a single point $D=\{x\}$ to simplify exposition. Define
	\begin{align*}
	W_j:=V_j \oplus V_{j-1} \oplus \dots \oplus V_1, \quad B_i= A_i \oplus A_{i-1} \oplus\dots \oplus A_1,
	\end{align*}
	which are vector spaces or trivial bundles on $X$. Then, let
	\begin{align*}
	\mathcal{W}_j:=E \times_{{\rm GL}_{\theta_x}} W_j.
	\end{align*}
	Clearly, we then have
	\begin{align*}
	\deg \mathcal{W}_j= \deg \mathcal{V}_j + \dots + \deg \mathcal{V}_1.
	\end{align*}
	Therefore, the following equation applies
	\begin{align*}
	\sum_{j=1}^r \lambda_j \deg \mathcal{V}_j=\lambda_r \deg \mathcal{W}_r + \sum_{j=1}^{r-1}(\lambda_j - \lambda_{j+1}) \deg \mathcal{W}_{j}.
	\end{align*}
	Note also that $\mathcal{W}_j$ can be realized as a vector bundle equipped with a natural parabolic structure.\\
	Similar calculations then provide the formulas
	\begin{align*}
	\sum_{i,j} \alpha_i \lambda_j \dim(A_i \cap V_j) &= \sum_{i=1}^l \left( \lambda_r \dim(A_i \cap W_r)  + \sum_{j=1}^{r-1} (\lambda_j-\lambda_{j+1}) \dim(A_i \cap W_j)  \right)\\
	&= \lambda_r \sum_{i=1}^{l-1} (\alpha_i-\alpha_{i+1}) \dim(B_i) +  \sum_{i=1}^{l-1} \sum_{j=1}^{r-1} (\alpha_i -\alpha_{i+1})(\lambda_j-\lambda_{j+1}) \dim(B_i \cap W_j)
	\end{align*}
	and
	\begin{align*}
	\langle \mu \varpi, \chi \rangle=\mu \sum_{j=1}^{r} \lambda_j \dim(V_j) = \mu \lambda_r \dim(W_r)+ \mu \sum_{j=1}^{r-1} (\lambda_j - \lambda_{j+1}) \dim(W_j) .
	\end{align*}
	Therefore, we have
	\begin{align*}
	parh \deg E(\varsigma,\chi)-\langle \mu \varpi, \chi \rangle = &\lambda_r \deg \mathcal{W}_r + \sum_{j=1}^{r-1}(\lambda_j - \lambda_{j+1}) \deg \mathcal{W}_{j}\\
	+& \lambda_r \sum_{i=1}^{l-1} (\alpha_i-\alpha_{i+1}) \dim(B_i) +  \sum_{i=1}^{l-1} \sum_{j=1}^{r-1} (\alpha_i -\alpha_{i+1})(\lambda_j-\lambda_{j+1}) \dim(B_i \cap W_j)\\
	-& \left(\mu \lambda_r \dim(W_r)+ \mu \sum_{j=1}^{r-1} (\lambda_j - \lambda_{j+1} ) \dim(W_j) \right)\\
	=& \sum_{r=1}^{j-1}(\lambda_j - \lambda_{j+1})(\deg \mathcal{W}_j + \sum_{i=1}^{l-1} (\alpha_i - \alpha_{i+1}) \dim(B_i \cap W_j) -\mu \dim(W_j) )\\
	+& \lambda_r \left( \deg \mathcal{W}_r +\sum_{i=1}^{l-1}(\alpha_i-\alpha_{i+1})\dim(B_i)-\mu \dim(W_r)  \right)\\
	=&\sum_{r=1}^{j-1}(\lambda_j - \lambda_{j+1}) (par \deg \mathcal{W}_j - \mu \dim(W_j))+\lambda_r (par \deg\mathcal{W}_r - \mu \dim (W_r))\\
	=& \sum_{r=1}^{j-1}(\lambda_j - \lambda_{j+1}) (par \deg \mathcal{W}_j - \mu \dim(W_j)).
	\end{align*}
	Note that $\lambda_j-\lambda_{j+1} < 0$. We thus conclude that $parh \deg E(\varsigma,\chi)-\langle \mu \varpi, \chi \rangle \geq 0$ if and only if for any parabolic subbundle $\mathcal{W}$ of $E$, one has
	\begin{align*}
	\frac{par \deg \mathcal{W}}{ rk (\mathcal{W})} \leq \frac{par \deg E}{ rk(E)}.
	\end{align*}
	This exactly means that $E$ is semistable as a parabolic bundle.
\end{proof}

\begin{rem}\label{rem parab deg}
	The above discussion shows that the parabolic degree and rank (and therefore the Hilbert polynomial) of the corresponding parabolic bundles are uniquely determined by $\boldsymbol\theta$ and $\mu$. 
\end{rem}

From Proposition \ref{505}, weights $\boldsymbol\theta$ of a parahoric torsor contribute to weights in the corresponding parabolic bundles. More precisely, if $\theta_x=\sum_{i=1}^n \frac{a_i}{d} t_i$, the weights in the filtration of the corresponding parabolic bundle on $x$ are exactly given by $\frac{a_i}{d}$  satisfying
\begin{align*}
0 \leq \frac{a_n}{d} \leq \dots \leq \frac{a_1}{d} < 1,
\end{align*}
which is the classical definition of a parabolic bundle \cite{MS}. Now if we consider an arbitrary weight (not necessary to be small), with the same argument as above, a parahoric $\mathcal{G}_{\boldsymbol\theta}$-torsor corresponds to a filtered bundle with weights, of which the weights are rational numbers (not necessarily in $[0,1)$). Abusing terminology, we still call such filtered bundles with weights \emph{parabolic bundles}, and the \emph{degree of parabolic bundle} is still defined as the sum of the degree of the underlying bundle together with the sum of weights. Therefore, we have the following corollary:

\begin{cor}\label{cor parah=parab}
	Let $\boldsymbol\theta$ be a collection of weights. Let $(E,\varphi)$ be a logahoric ${\rm GL}_{\boldsymbol\theta}$-Higgs torsor. Denote by $\mu:=\frac{par\deg E}{rk (E)}$ the parabolic slope of $E$ as a parabolic bundle. Then, $(E,\varphi)$ is $R_\mu$-stable (resp. $R_{\mu}$-semistable) if and only if $E$ is stable (resp. semistable) as a parabolic Higgs bundle. 
\end{cor}

\subsection{$R_\mu$-stability Condition for General Reductive Groups}

In this section, we will define a stability condition for parahoric $\mathcal{G}_{\boldsymbol\theta}$-torsors for a general complex reductive group $G$. In particular, in the case when $G=\mathrm{GL}_n$ we will recover the $R_\mu$-stability condition of the previous subsection. We have similarly a definition of $R_\mu$-stability for general $G$ depending on a choice of $\mu\in\mathfrak{t}$.

\begin{defn}\label{Ralpha}
	Fixing an element $\mu\in \mathfrak{t}$ in the Lie algebra of a maximal torus,
	a parahoric $\mathcal{G}_{\boldsymbol\theta}$-torsor $E$ on $X$ is called \textit{$R_\mu$-stable} (resp. \textit{$R_\mu$-semistable}) if for
	\begin{itemize}
		\item any proper parabolic subgroup $P$ of $G$,
		\item any reduction $\varsigma:X\to E/\mathcal{P}_{\boldsymbol\theta}$,
		\item any nontrivial anti-dominant character $\chi: \mathcal{P}_{\boldsymbol\theta} \rightarrow \mathbb{G}_m$ (not necessarily trivial on the center $\mathfrak{z}$),
	\end{itemize}
	one has
	\[parh\deg E(\varsigma,\chi) -\langle \mu,\chi \rangle> 0\quad (\mbox{resp. }\ge 0).\]
\end{defn}

The relation between these two definitions is given by the following choice of $\mu$:
\begin{prop}\label{alphaexst}
	Let $E$ be a parahoric $\mathcal{G}_{\boldsymbol\theta}$-torsor. There exists a canonical choice of $\mu\in \mathfrak{z}$ in the center of $\mathfrak{t}$, depending on the topological type of $E$, such that $E$ is $R$-stable (resp. $R$-semistable) if and only if $E$ is $R_\mu$-stable (resp. $R_\mu$-semistable).
\end{prop}

\begin{proof}
	Since when $\chi$ is anti-dominant and trivial on the center $\mathfrak{z}$ there is nothing to prove, we only need to find the pairing of $\mu$ with $\chi$ nontrivial in $\mathfrak{z}^*$. We find a base of $\mathfrak{z}^*$, say given by $\chi_1,\ldots,\chi_n$. By abuse of notation we can also think of these elements as anti-dominant characters.
	
	Here we view $\chi_i: \mathcal{G}_{\boldsymbol\theta}\to \mathbb{G}_m$ and $\varsigma_0:X\to E/\mathcal{G}_{\boldsymbol\theta}$ the parabolic reduction when we choose parabolic subgroup to be $G$ itself. Therefore we may define $\mu$ by the action on $\chi_i$:
	\[\langle \mu, \chi_i\rangle:=parh\deg E(\varsigma_0,\chi_i).\]
	In this case, the parahoric degree is nothing but
	\[parh\deg E(\varsigma_0,\chi_i) = \deg L(\varsigma_0,\chi_i) + \langle \boldsymbol\theta,\chi_i\rangle,\]
	where $L(\varsigma_0,\chi_i)$ is the line bundle on $X$ we defined in \S 4.1, and thus the definition of $\mu$ depends only on the topology of $E$. Then for every combination $\chi=\sum a_i\chi_i + \delta$, where $\delta$ is an anti-dominant element that acts trivially on the center, we have
	\[
	parh\deg E(\varsigma,\chi)-\langle\mu,\chi\rangle =\sum_{i=1}^n a_i\left(parh\deg E(\varsigma, \chi_i)-\langle \mu,\chi_i \rangle \right)+parh\deg E(\varsigma,\delta).
	\]
	We will show that
	\begin{align*}
	parh\deg E(\varsigma, \chi_i)-\langle \mu,\chi_i \rangle=0,
	\end{align*}
	for each $i$. Thus the positivity of the number $parh\deg E(\varsigma,\chi)-\langle\mu,\chi\rangle$ depends only on the positivity of $parh\deg E(\varsigma,\delta)$. This will imply the equivalence of the stability (semistability) conditions.
	
	Thus the only thing remaining is that $parh\deg E(\varsigma_0,\chi_i)=parh\deg E(\varsigma,\chi_i)$ for every reduction $\varsigma: X\to E/ \mathcal{P}_{\boldsymbol\theta}$. This is because the characters $\chi_i:\mathcal{P}_{\boldsymbol\theta} \to\mathbb{G}_m$ coming from elements in $\mathfrak{z}$ can be lifted to the same ones $\chi_i: \mathcal{G}_{\boldsymbol\theta}\to\mathbb{G}_m$. Therefore, we have
	\begin{align*}
	L(\varsigma_0,\chi_i)\cong L(\varsigma,\chi_i),
	\end{align*}
	and thus
	\begin{align*}
	parh \deg E(\varsigma_0,\chi_i)=parh \deg E(\varsigma,\chi_i).
	\end{align*}
	This completes the proof.
\end{proof}

\begin{rem}
	When $G=\mathrm{GL}_n$, we may find the direct equivalence of $R_\mu$-stability and $R$-stability by Proposition \ref{alphaexst}. Note that we only have one generator $\chi_1$ of $\mathfrak{z}$ given by the diagonal matrix, which also satisfies $\langle \varpi,\chi_1 \rangle=n$ (see \S 5.2). With the same idea as in Proposition \ref{alphaexst}, we need to find the element $\mu \varpi$, where $\mu$ is a rational number, such that
	\begin{align*}
	\langle \mu \varpi,\chi_1 \rangle= parh\deg E(\varsigma_0,\chi_1).
	\end{align*}
	With the same calculation as in formula (\ref{ast}), we find
	\[
	parh\deg E(\varsigma_0,\chi_1)=par \mathrm{deg}(E).
	\]
	Therefore, $\mu=\frac{par \deg(E)}{n}$. This recovers the definition of $\mu$ in Proposition \ref{505}.
\end{rem}

\begin{defn}
	A parahoric $\mathcal{G}_{\boldsymbol\theta}$-torsor is of \emph{type $\mu$}, if $\mu$ is the element given in Proposition \ref{alphaexst}.
\end{defn}

Let $G_1$ and $G_2$ be two connected complex reductive groups with maximal torus $T_1$ and $T_2$ respectively. Denote by $\mathfrak{t}_1$ and $\mathfrak{t}_2$ the corresponding Lie algebras. Let $\boldsymbol\theta_1$ be a collection of weights for $G_1$, and $\mu_1$ an element in $\mathfrak{t}_1$. Given a homomorphism $\varrho: G_1 \rightarrow G_2$, denote by $\boldsymbol\theta_2$ the corresponding weights of $\boldsymbol\theta_1$ for $G_2$ and let $\mu_2$ be the corresponding element of $\mu_1$ in $\mathfrak{t}_2$.

\begin{prop}\label{prop G1 and G2}
	Given a homomorphism $\varrho: G_1 \rightarrow G_2$, suppose that $\varrho(Z_0(G_1)) \subseteq Z_0(G_2)$ and the induced homomorphism $G_1/Z_0(G_1) \rightarrow G_2/Z_0(G_2)$ is finite, where $Z_0(G_i)$ is the connected component of the center $Z(G_i)$ containing the identity. Let $E_1$ be a parahoric $(\mathcal{G}_1)_{\boldsymbol\theta_1}$-torsor of type $\mu_1$. The parahoric torsor $E_1$ is $R_{\mu_1}$-semistable if and only if the parahoric $(\mathcal{G}_2)_{\boldsymbol\theta_2}$-torsor $\varrho_* E_1$ is $R_{\mu_2}$-semistable.
\end{prop}

\begin{proof}
	By Theorem \ref{415}, it is equivalent to consider the statement for $(\Gamma,G)$-bundles. Then, the proof is the same as in \cite[Proposition 3.17]{Rama19961}.
\end{proof}

\begin{rem}\label{rem adj bundle}
	As a special case of the above proposition, we consider the adjoint representation $G \rightarrow {\rm GL}(\mathfrak{g})$. By Proposition \ref{alphaexst}, the element $\mu$ lies in the center $\mathfrak{z}$. Therefore, the corresponding topological element of $\mu$ for the adjoint bundle is zero, i.e. the trivial element. Although by Propositions \ref{505}, \ref{alphaexst} and \ref{prop G1 and G2} we know that a parahoric torsor $E$ is $R_\mu$-semistable if and only if $E(\mathfrak{g})$ is semistable (as a parabolic bundle), the information of $\mu$ is lost during this correspondence. In order to recover the lost data $\mu$, we note that there is a natural projection $G \rightarrow G/G' \times G/Z$, where $G'$ is the derived group and $Z$ is the center. The quotient group $G/G'$ is a torus and the morphism $G/Z={\rm Ad}(G) \hookrightarrow {\rm GL}(\mathfrak{g})$ is injective. Therefore, the lost data $\mu$ can be recovered from a unique element in $G/G'$. We apply this idea to give the construction of the moduli space in \S\ref{sect_moduli_space}.
\end{rem}

\section{Moduli Space of Logahoric Higgs Torsors}\label{sect_moduli_space}
We proceed next with the construction of the moduli space of $R$-semistable logahoric $\mathcal{G}_{\boldsymbol\theta}$-Higgs torsors (Theorem \ref{701}). Although we only give the construction of the moduli space in the case of Higgs bundles, our approach also works for logahoric connections (see Remark \ref{moduli of local system}).

\subsection{Main Result}
We define the moduli problem of $R$-semistable logahoric $\mathcal{G}_{\boldsymbol\theta}$-Higgs torsors of type $\mu$ on $X$
\begin{align*}
\mathcal{M}^{Rss}_H(X, \mathcal{G}_{\boldsymbol\theta},\mu): {\rm (Sch/\mathbb{C})}^{\rm op}  \rightarrow {\rm Sets}
\end{align*}
as follows. For each $\mathbb{C}$-scheme $S$, the set $\widetilde{\mathcal{M}}^{Rss}_H(X,\mathcal{G}_{\boldsymbol\theta},\mu)(S)$ is defined as the collection of pairs $(E,\varphi)$ up to isomorphism such that
\begin{itemize}
	\item $E$ is an $S$-flat family of parahoric $\mathcal{G}_{\boldsymbol\theta}$-torsors on $X$;
	\item $\phi: X_S \rightarrow E(\mathfrak{g}) \otimes \pi^*_X K_X(D)$ is a section, where $\pi_X: X_S \cong X \times S \rightarrow X$ is the natural projection;
	\item for each point $s \in S$, the restriction $(E|_{X \times s},\varphi|_{X \times s})$ is an $R$-semistable logahoric $\mathcal{G}_{\boldsymbol\theta}$-Higgs torsor of type $\mu$ on $X$.
\end{itemize}

The main theorem in this section is the following:
\begin{thm}\label{701}
	There exists a quasi-projective scheme $\mathfrak{M}^{Rss}_H(X,\mathcal{G}_{\boldsymbol\theta},\mu)$ as the moduli space for the moduli problem $\mathcal{M}^{Rss}_H(X,\mathcal{G}_{\boldsymbol\theta},\mu)$ of $R$-semistable logahoric $\mathcal{G}_{\boldsymbol\theta}$-Higgs torsors, and the geometric points of $\mathfrak{M}^{Rss}_H(X,\mathcal{G}_{\boldsymbol\theta},\mu)$ represent $S$-equivalence classes of $R$-semistable logahoric $\mathcal{G}_{\boldsymbol\theta}$-Higgs torsors of type $\mu$. Furthermore, there is an open subset $\mathfrak{M}^{Rs}_H(X,\mathcal{G}_{\boldsymbol\theta},\mu) \subseteq \mathfrak{M}^{Rss}_H(X,\mathcal{G}_{\boldsymbol\theta},\mu)$ parameterizing isomorphism classes of $R$-stable logahoric $\mathcal{G}_{\boldsymbol\theta}$-Higgs torsors of type $\mu$.
\end{thm}

To simplify terminology, we shall say \emph{moduli space of semistable objects} and omit the reference to $S$-equivalence classes. We prove the main theorem using the following steps for $R$-semistable logahoric Higgs torsors, and the statement for $R$-stable ones follows directly.
\begin{enumerate}
	\item There is a natural way to define the moduli problem $\mathcal{M}^{Rss}_{H,\boldsymbol\theta}(Y,\Gamma,G)$ of $R$-semistable logarithmic $(\Gamma,G)$-Higgs bundles of type $\boldsymbol\theta$ on $Y$. By Theorem \ref{305} and Theorem \ref{416}, we have
	\begin{align*}
	\mathcal{M}^{Rss}_H(X,\mathcal{G}_{\boldsymbol\theta}) \cong \mathcal{M}^{Rss}_{H,\boldsymbol\theta}(Y,\Gamma,G).
	\end{align*}
	\noindent Therefore, it is equivalent to construct the moduli space of $R$-semistable logarithmic $(\Gamma,G)$-Higgs bundles of type $\boldsymbol\theta$.
	Remember that we consider $X=Y/\Gamma$, and $X$ can be understood as the coarse moduli space of the stack $[Y/\Gamma]$. By definition, $\Gamma$-equivariant bundles on $Y$ are exactly bundles on the stack $[Y/\Gamma]$. Then, we have the following correspondence
	\begin{center}
		\begin{tikzcd}
		\text{$G$-Higgs bundles on $[Y/\Gamma]$} \arrow[r, leftrightarrow] &  \text{$(\Gamma,G)$-Higgs bundles on $Y$}.
		\end{tikzcd}
	\end{center}
	With respect to the above correspondence, it is enough to construct the moduli space of logarithmic $G$-Higgs bundles on $[Y/\Gamma]$.
	
	The above discussion shows that constructing the moduli space of $R$-semistable logahoric $\mathcal{G}_{\boldsymbol\theta}$-Higgs torsors on $X$ is equivalent to constructing the moduli space $\mathfrak{M}^{Rss}_{H,\boldsymbol\theta}([Y/\Gamma],G)$ of $R$-semistable logarithmic $G$-Higgs bundles on $[Y/\Gamma]$, and the following steps are devoted to the construction of the moduli space $\mathfrak{M}^{Rss}_{H,\boldsymbol\theta}([Y/\Gamma],G)$.
	
	\item A $(\Gamma,G)$-bundle $F$ is $R$-semistable if and only if $F$ is $R_\mu$-semistable, where $\mu$ is given in \S\ref{sect_R_mu_stab}. Considering the adjoint representation $G \rightarrow {\rm GL}(\mathfrak{g})$, the $\Gamma$-equivariant adjoint bundle $F(\mathfrak{g})$ is then semistable (see Remark \ref{rem adj bundle}), where we take $F(\mathfrak{g})$ as a parabolic bundle and the stability condition is given by the parabolic degree. In total, $F$ is $R$-semistable if the $\Gamma$-equivariant adjoint bundle $F(\mathfrak{g})$ is semistable. A similar statement for principal bundles is given by Ramanathan (see \cite[Corollary 3.18]{Rama19961} and \cite[Theorem 2.2]{FriMor1998}). Also, the above discussion can be generalized to logarithmic $(\Gamma,G)$-Higgs bundles directly. Now denote by $P$ the Hilbert polynomial of $F(\mathfrak{g})$. We only have to construct the moduli space $\mathfrak{M}^{ss}_H([Y/\Gamma],P)$ of semistable Higgs bundles on $[Y/\Gamma]$ with Hilbert polynomial $P$, where the stability condition is equivalent to the slope stability condition of  the corresponding parabolic bundles.
	
	\item Nironi in \cite{Nir} defined the $\mathcal{E}$-stability condition for sheaves on $[Y/\Gamma]$, where $\mathcal{E}$ is a generating sheaf. With a good choice of $\mathcal{E}$, the $\mathcal{E}$-stability condition for bundles on $[Y/\Gamma]$ is equivalent to the stability of $\Gamma$-equivariant bundles on $Y$, and therefore equivalent to the stability condition of the corresponding parabolic bundles on $X$. In total, the existence of the moduli space of $\mathcal{E}$-semistable logarithmic Higgs bundles on $[Y/\Gamma]$ will imply the existence of the moduli space $\mathfrak{M}^{ss}_H([Y/\Gamma],P)$ in the second step.
	
	Note, lastly, that the construction of the moduli space of Higgs bundles on smooth projective varieties (over $\mathbb{C}$) was first given by Simpson \cite{Simp2}, and this construction was generalized by the second author in \cite{Sun202003} to the case of (tame) projective Deligne--Mumford stacks (over any algebraically closed field); we refer the reader to the aforementioned articles for further information.
\end{enumerate}
\noindent The above discussion is the basic idea to construct the moduli space $\mathfrak{M}^{Rss}_H(X,\mathcal{G}_{\boldsymbol\theta},\mu)$.

In \S 6.2 below, we construct first the moduli space $\mathfrak{M}^{\mathcal{E}ss}_{H}([Y/\Gamma],P)$ of $\mathcal{E}$-semistable logarithmic Higgs bundles on $[Y/\Gamma]$ with a given (modified) Hilbert polynomial $P$. With a good choice of $\mathcal{E}$, the stability condition for logarithmic Higgs bundles on $[Y/\Gamma]$ coincides with the stability condition for parabolic Higgs bundles on $X$. Therefore, we obtain the moduli space $\mathfrak{M}^{ss}_{H}([Y/\Gamma],P)$ of semistable adjoint Higgs bundles $(F(\mathfrak{g}),\phi)$. In \S 6.3, we start with $\mathfrak{M}^{ss}_{H}([Y/\Gamma],P)$ and construct the moduli space $\mathfrak{M}^{Rss}_{H,\boldsymbol\theta}([Y/\Gamma], G,\mu)$ of $R$-semistable logarithmic $G$-Higgs bundles (of type $\boldsymbol\theta$ and $\mu$) on $[Y/\Gamma]$. As we discussed above, the moduli space $\mathfrak{M}^{Rss}_{H,\boldsymbol\theta}([Y/\Gamma], G ,\mu)$ is exactly the moduli space $\mathfrak{M}^{Rss}_{H}(X,\mathcal{G}_{\boldsymbol\theta},\mu)$ of $R$-semistable logahoric $\mathcal{G}_{\boldsymbol\theta}$-Higgs bundles of type $\mu$ on $X$. The approach follows from \cite{Rama19961,Rama19962}, and we also refer the reader to \cite{Hye} for more details.

\subsection{Moduli Space of Logarithmic Higgs Bundles on Quotient Stacks}
In this section, we give the construction of the moduli space $\mathfrak{M}^{\mathcal{E}ss}_H([Y/\Gamma])$ of $\mathcal{E}$-semistable logarithmic Higgs bundles on quotient stacks $[Y/\Gamma]$, where $\mathcal{E}$ is a generating sheaf.

Note that a Higgs field $\phi: F \rightarrow F \otimes \Omega^1_{\mathcal{X}}(R)$ is equivalent to a morphism ${\rm Sym}(T_{\mathcal{X}}(-R)) \rightarrow \mathcal{E}nd(F)$. Denote by $\Lambda={\rm Sym}(T_{\mathcal{X}}(-R))$ the corresponding sheaf of differential graded algebras (see \cite{Simp2,Sun202003}). Therefore, Higgs bundles are a special case of $\Lambda$-modules, and the moduli space of $\mathcal{E}$-semistable Higgs bundles (that is, ${\rm GL}$-Higgs bundles) $\mathfrak{M}^{\mathcal{E}ss}_H([Y/\Gamma])$ on $[Y/\Gamma]$ is constructed in the same way as the moduli space of $\Lambda$-modules. We will briefly review the construction of the moduli space of $\Lambda$-modules and refer the reader to \cite{Sun202003} for more details.

Let $\mathcal{X}:=[Y/\Gamma]$ be the quotient stack, and denote by $\pi: \mathcal{X} \rightarrow X$ the coarse moduli space. A  locally free sheaf $\mathcal{E}$ is a \emph{generating sheaf}, if for any coherent sheaf $F$ on $\mathcal{X}$, the morphism
\begin{align*}
\theta_{\mathcal{E}}(F) : \pi^* \pi_* \mathcal{H}om(\mathcal{E},F) \otimes \mathcal{E} \rightarrow F
\end{align*}
is surjective. By \cite[Proposition 5.2]{OlSt}, there exists a generating sheaf $\mathcal{E}$ for $\mathcal{X}$ in our case. A very important property of the generating sheaf is that the functor
\begin{align*}
F_{\mathcal{E}}: {\rm Coh}(\mathcal{X}) & \rightarrow {\rm Coh}(X)\\
F & \mapsto \pi_* \mathcal{H}om(\mathcal{E},F)
\end{align*}
induces a closed immersion of quot-schemes (see \cite[Lemma 6.2]{OlSt})
\begin{align*}
F_{\mathcal{E}}: {\rm Quot}(G,\mathcal{X},P) & \rightarrow {\rm Quot}(F_{\mathcal{E}}(G),X,P)\\
[G \rightarrow F] & \mapsto [F_{\mathcal{E}}(G) \rightarrow F_{\mathcal{E}}(F)],
\end{align*}
where $G$ is a coherent sheaf on $\mathcal{X}$ and $P$ is an integer polynomial as the ``Hilbert polynomial". This property implies that ${\rm Quot}(G,\mathcal{X},P)$ is a projective scheme. Therefore, we can construct the moduli space of coherent sheaves on $\mathcal{X}$ with respect to a ``good" stability condition. This ``good" stability condition is called \emph{$\mathcal{E}$-stability}. First, we define the \emph{modified Hilbert polynomial}. Let $F$ be a coherent sheaf on $\mathcal{X}$. The \emph{modified Hilbert polynomial} $P_{\mathcal{E}}(F,m)$ is defined as
\begin{align*}
P_{\mathcal{E}}(F,m)=\chi(\mathcal{X},F \otimes \mathcal{E}^{\vee} \otimes \pi^* \mathcal{O}_X(m)), \quad m \gg 0.
\end{align*}

\begin{defn}\label{702}
	A pure coherent sheaf $F$ on $\mathcal{X}$ is \emph{$\mathcal{E}$-semistable} (resp. \emph{$\mathcal{E}$-stable}), if for every proper subsheaf $F' \subseteq F$ we have
	\begin{align*}
	p_{\mathcal{E}}(F') \leq p_{\mathcal{E}}(F)\quad (\text{resp. } p_{\mathcal{E}}(F') < p_{\mathcal{E}}(F)),
	\end{align*}
	where $p_{\mathcal{E}}(\bullet)$ is the reduced modified Hilbert polynomial.
\end{defn}

Let $\Lambda$ be a sheaf of graded algebras on $\mathcal{X}$. A \emph{coherent $\Lambda$-sheaf $F$} is a coherent sheaf (with respect to the $\mathcal{O}_{\mathcal{X}}$-structure) on $\mathcal{X}$ together with a left $\Lambda$-action. A subsheaf $F' \subseteq F$ is a \emph{$\Lambda$-subsheaf}, if we have $\Lambda \otimes F' \subseteq F'$. There are several ways to understand ``an action of $\Lambda$". Usually an action of $\Lambda$ on $F$ means that we have a morphism
\begin{align*}
\Lambda \rightarrow \mathcal{E}nd(F).
\end{align*}
Equivalently, this morphism can be interpreted as
\begin{align*}
\Lambda \otimes F \rightarrow F.
\end{align*}
The above morphism induces a morphism ${\rm Gr}_1(\Lambda) \otimes F \rightarrow F$ naturally. If ${\rm Gr}_1(\Lambda)$ is a locally free sheaf, then it corresponds to a morphism $F \rightarrow F \otimes {\rm Gr}_1(\Lambda)^*$.

\begin{defn}\label{703}
	A $\Lambda$-sheaf $F$ is \emph{$\mathcal{E}$-semistable} (resp. \emph{$\mathcal{E}$-stable}), if $F$ is a pure coherent sheaf and for any $\Lambda$-subsheaf $F' \subseteq F$ with $0 < {\rm rk}(F') < {\rm rk}(F)$, we have
	\begin{align*}
	p_{\mathcal{E}}(F') \leq p_{\mathcal{E}}(F),\quad (\text{resp. }<).
	\end{align*}
\end{defn}

Now we are ready to construct the moduli space of $p$-semistable $\Lambda$-sheaves. Let $k$ be a positive integer. We consider the quot-scheme $Q_1:={\rm Quot}(\Lambda_k \otimes V \otimes G,\mathcal{X},P)$, which parameterizes quotients $[\Lambda_k \otimes V \otimes G \rightarrow F]$ such that
\begin{itemize}
	\item $P$ is an integer polynomial taken as the modified Hilbert polynomial;
	\item $V$ is a $\mathbb{C}$-vector space of dimension $P(N)$, where $N$ is a large enough positive integer;
	\item $G$ is $\pi^*\mathcal{O}_X(-N)$.
\end{itemize}
We construct the moduli space following the steps below.
\begin{enumerate}
	\item There exists a closed subscheme $Q_2 \subseteq Q_1$ such that any point $[\rho: \Lambda_k \otimes V \otimes G \rightarrow F]$ \emph{admits a factorization}. More precisely, the quotient $\rho$ has the following factorization
	\begin{center}
		\begin{tikzcd}
		\Lambda_k \otimes V \otimes G \arrow[rd, "1 \otimes \rho'"] \arrow[rr, "\rho"] &  & \mathcal{F}  \\
		& \Lambda_k \otimes F \arrow[ru,"\phi_k"]&
		\end{tikzcd}
	\end{center}
	such that
	\begin{itemize}
		\item the induced morphism $\rho':V \otimes G \rightarrow F$ is an element in ${\rm Quot}(V \otimes G,\mathcal{X},P)$;
		\item $\phi_k: \Lambda_k \otimes F \rightarrow F$ is a morphism.
	\end{itemize}
	This condition gives a $\Lambda_k$-structure on the coherent sheaf $F$.
	\item Let $[\rho: \Lambda_k \otimes V \otimes G \rightarrow F] \in Q_2$ be a point. Denote by $[\rho': V \otimes G \rightarrow F]$ the quotient in the factorization of $\rho$. The morphism $\rho$ also induces a morphism $\Lambda_1 \otimes V \otimes G \rightarrow F$. Denote by $K$ the kernel of the quotient map
	\begin{align*}
	0 \rightarrow K \rightarrow V \otimes G \rightarrow F \rightarrow 0.
	\end{align*}
	Denote by $Q_3 \subseteq Q_2$ the closed subscheme of $Q_2$ such that the induced map $\Lambda_1 \otimes K \rightarrow \Lambda_1 \otimes V \otimes G \rightarrow F$ is trivial.
	\item As discussed in the last step, a point $[\rho: \Lambda_k \otimes V \otimes G \rightarrow F] \in Q_3$ induces a morphism $\Lambda_1 \otimes F \rightarrow F$. This morphism also induces the following ones
	\begin{align*}
	\underbrace{(\Lambda_1 \otimes \dots \otimes \Lambda_1)}_j \otimes F \rightarrow F,
	\end{align*}
	for each positive integer $j$. Denote by $K_j$ the kernel of the surjection
	\begin{align*}
	\underbrace{\Lambda_1 \otimes \dots \otimes \Lambda_1}_j \rightarrow \Lambda_j \rightarrow 0.
	\end{align*}
	Thus, we have a natural map
	\begin{align*}
	K_j \otimes F \rightarrow F.
	\end{align*}
	For each positive integer $j$, there exists a closed subscheme $Q_{4,j} \subseteq Q_3$ such that if $[\rho] \in Q_{4,j}$, then the corresponding map $K_j \otimes F \rightarrow F$ is trivial.
	\item Denote by $Q_{4,\infty}$ the intersection of $Q_{4,j}$, for $j \geq 1$.
	\item The above steps show that a quotient $[\rho] \in Q_{4,\infty}$ gives a $\Lambda$-structure on $F$, and this $\Lambda$-structure will induce a $\Lambda_k$-structure on $F$. Note that this induced $\Lambda_k$-structure may not be the same as the given one. However, there is a closed subscheme $Q_5 \subseteq Q_{4,\infty}$ such that these two structures are the same.
	\item Let $Q_6 \subseteq Q_5$ be the open subscheme such that if $[\rho: \Lambda_k \otimes V \otimes G \rightarrow F] \in Q_6$, then we have $V \cong H^0(Y,F(N))$.
	\item There is an open subset $Q^{ss}_{\Lambda} \subseteq Q_6$ such that if $[\rho: \Lambda_k \otimes V \otimes G \rightarrow F] \in Q^{ss}_{\Lambda}$, then $F$ is a $p$-semistable $\Lambda$-sheaf.
\end{enumerate}
With respect to the above construction, the subset
\begin{align*}
Q^{ss}_{\Lambda} \subseteq {\rm Quot}(\Lambda_k \otimes V \otimes G,\mathcal{X},P)
\end{align*}
is a quasi-projective scheme, which parameterizes $\Lambda$-modules with modified Hilbert polynomial $P$. There is an induced ${\rm SL}(V)$-action on $Q^{ss}_{\Lambda}$. Given a point \begin{align*}
[\rho: \Lambda_k \otimes V \otimes G \rightarrow F] \in {\rm Quot}(\Lambda_k \otimes V \otimes G,\mathcal{X},P),
\end{align*}
if $F$ is $\mathcal{E}$-semistable (resp. $\mathcal{E}$-stable), then it is also semistable (resp. stable) in the sense of GIT \cite[Lemma 6.4]{Sun202003}. Define $\mathfrak{M}^{\mathcal{E}ss}_{\Lambda}(\mathcal{X},P):=Q^{ss}_{\Lambda}/{\rm SL}(V)$.

\begin{thm}[Theorem 6.8 in \cite{Sun202003}]\label{704}
	The quasi-projective scheme $\mathfrak{M}^{\mathcal{E}ss}_{\Lambda}(\mathcal{X},P)$
	is the coarse moduli space of $\mathcal{E}$-semistable $\Lambda$-sheaves with modified Hilbert polynomial $P$ on $\mathcal{X}$, and the geometric points of $\mathfrak{M}^{\mathcal{E}ss}_{\Lambda}(\mathcal{X},P)$ represent the equivalence classes of $\mathcal{E}$-semistable $\Lambda$-sheaves, where the equivalence is given by the Jordan--H\"older filtration and is known as $S$-equivalence.
\end{thm}

There is a well-known correspondence between $\Gamma$-equivariant bundles on $Y$ and parabolic bundles on the coarse moduli space $X$ \cite{NaSt}. It was observed by Nironi in \cite{Nir}, that with a good choice of the generating sheaf $\mathcal{E}$, the $\mathcal{E}$-stability of coherent sheaves on $[Y/\Gamma]$ is equivalent to the stability of the corresponding parabolic bundle on $X$.

\begin{lem}[\S 7.2 in \cite{Nir}]
	There exists a generating sheaf $\mathcal{E}$ on $[Y/\Gamma]$, such that the $\mathcal{E}$-stability of coherent sheaves on $[Y/\Gamma]$ is equivalent to the stability of the corresponding parabolic bundles on $X$, and therefore equivalent to the stability of the corresponding $\Gamma$-equivariant bundles on $Y$.
\end{lem}

\begin{cor}\label{707}
	There exists a generating sheaf $\mathcal{E}$ such that the moduli space of $\mathcal{E}$-semistable logarithmic Higgs bundles on $\mathcal{X}$ is isomorphic to the moduli space of semistable $\Gamma$-equivariant logarithmic Higgs bundles on $Y$, that is,
	\begin{align*}
	\mathfrak{M}^{\mathcal{E}ss}_{H}(\mathcal{X},P) \cong \mathfrak{M}^{ss}_{H}([Y\backslash \Gamma],P).
	\end{align*}
\end{cor}

\subsection{Moduli Space of R-semistable Equivariant Logarithmic G-Higgs Bundles}
We first review two important lemmas. Let $A, A'$  be two algebraic groups, and let $\rho: A' \rightarrow A$ be a homomorphism. Let $\mathcal{S}$ be a set of isomorphism classes of principal $A$-bundles on $Y$. Let $\mathscr{E} \rightarrow \mathcal{R} \times X$ be a family of principal $A$-bundles in $\mathcal{S}$. Suppose that an algebraic group $H$ acts on $T$ by $\sigma: H \times T \rightarrow T$, and we have an isomorphism $\widetilde{\sigma}: H \times \mathscr{E} \cong (\sigma \times id_X)^* \mathscr{E}$. The family $\mathscr{E}$ is an \emph{$H$-universal family}, if the following conditions hold:
\begin{enumerate}
	\item For any family of principal $A$-bundles $\mathscr{F} \rightarrow S \times X$ and any point $s \in S$, there exists an open neighbourhood of $s \in S$, and a morphism $f: U \rightarrow \mathcal{R}$ such that $\mathscr{F}|_{U \times X} \cong (f \times 1_X)^* \mathscr{E}$.
	\item Given two morphisms $f_1, f_2: S \rightarrow \mathcal{R}$ and an isomorphism $\varphi: \mathscr{E}_{f_1} \cong \mathscr{E}_{f_2}$, there exists a unique morphism $h: S \rightarrow H$ such that $f_2= \sigma \circ (f_1 \times h)$ and $\varphi=(f_1 \times h \times 1_X)^*(\widetilde{\sigma})$.
\end{enumerate}

We consider the functor
\begin{align*}
\widetilde{\Gamma}(\rho,\mathscr{E}): ({\rm Sch}/\mathcal{R}) \rightarrow {\rm Sets},
\end{align*}
such that for each $\mathcal{R}$-scheme $S$, $\widetilde{\Gamma}(\rho,\mathscr{E})(S)$ is the set of pairs $(E,\tau)$, where $E$ is a principal $A'$-bundles on $S \times X$, and $\tau: \rho_*(E) \rightarrow \mathscr{E}_S$ is an isomorphism.

\begin{lem}[Lemma 4.8.1 in \cite{Rama19962}]\label{708}
	If $\rho: A' \rightarrow A$ is injective, then the functor $\widetilde{\Gamma}(\rho,\mathscr{E})$ is representable by a quasi-projective $\mathcal{R}$-scheme.
\end{lem}
Denote by $\mathcal{R}_1$ the quasi-projective $\mathcal{R}$-scheme representing $\widetilde{\Gamma}(\rho,\mathscr{E})$, and $\mathcal{E}_1 \rightarrow \mathcal{R}_1 \times X$ the universal family corresponding to the universal element in $\widetilde{\Gamma}(\rho,\mathscr{E})(\mathcal{R}_1)$.

\begin{lem}[Lemma 4.10 in \cite{Rama19962}]\label{709}
	Let $\rho: A' \rightarrow A$ be a homomorphism of algebraic groups. Let $\mathscr{E} \rightarrow \mathcal{R} \times X$ be an $H$-universal family for a set $\mathcal{S}$ of principal $G$-bundles. Suppose that the functor $\widetilde{\Gamma}(\rho,\mathscr{E})$ is representable by a scheme $\mathcal{R}_1$. Then, we have
	\begin{enumerate}
		\item The group $H$ acts on $\mathcal{R}_1$ in a natural way, and $\mathcal{R}_1$ is an $H$-universal family for the set of principal $A'$-bundles, which give $A$-bundles in $\mathcal{S}$ by extending the structure group by $\rho: A' \rightarrow A$.
		\item If $\rho$ is injective, then there exists a universal family $\mathscr{E}_1 \rightarrow \mathcal{R}_1 \times X$ of principal $A'$-bundles, which corresponds to the universal element in $\widetilde{\Gamma}(\rho,\mathscr{E})(\mathcal{R}_1)$.
	\end{enumerate}
\end{lem}

Now we will construct the moduli space $\mathfrak{M}^{Rss}_{H,\boldsymbol\theta}([Y \backslash \Gamma],G)$ of $R$-semistable logarithmic $(\Gamma,G)$-Higgs bundles on $Y$ (with arbitrary $\mu$). Let $G \rightarrow {\rm GL}(\mathfrak{g})$ be the adjoint representation. Then, for every logarithmic $(\Gamma,G)$-Higgs bundle $(F,\phi)$ on $Y$, we can associate an adjoint Higgs bundle $(F(\mathfrak{g}),\phi)$. By Theorem \ref{704} and Corollary \ref{707}, there exists a moduli space of semistable $\Gamma$-equivariant (adjoint) Higgs bundles on $Y$ with Hilbert polynomial $P$. Denote by $\mathcal{R}:=\mathfrak{M}^{Rss}_{H,\boldsymbol\theta}([Y \backslash \Gamma],P)$ the moduli space and $\mathscr{E} \rightarrow \mathcal{R} \times Y$ the universal family, and there is a natural ${\rm GL}(V)$-action on $\mathcal{R}$ for some vector space $V$. Now we follow Ramanathan's approach \cite{Rama19962} to construct the moduli space $\mathfrak{M}^{Rss}_{H,\boldsymbol\theta}([Y \backslash \Gamma],G)$. We omit the fixed Hilbert polynomial $P$ for simplicity.
\begin{center}
	\begin{tikzcd}
	{\rm Aut}(\mathfrak{g}) \arrow[r] & {\rm Aut}(\mathfrak{g}) \times \mathbb{C}^* \arrow[r, hook] & {\rm GL}(\mathfrak{g})\\
	{\rm Ad}(G)=G/Z \arrow[u, hook] & G/G' \times G/Z \arrow[l] & G \arrow[l]
	\end{tikzcd}
\end{center}

\begin{enumerate}
	\item Let $\mathbb{C}^* \times {\rm Aut}(\mathfrak{g}) \hookrightarrow {\rm GL}(\mathfrak{g})$ be the natural inclusion. By Lemma \ref{709}, we get a universal family for the set of $\Gamma$-equivariant logarithmic $(\mathbb{C}^* \times {\rm Aut}(\mathfrak{g}))$-Higgs bundles, of which the associated $\Gamma$-equivariant logarithmic ${\rm GL}(\mathfrak{g})$-Higgs bundles are semistable. Denote by
	\begin{align*}
	\mathscr{E}_1 \rightarrow \mathcal{R}_1 \times Y
	\end{align*}
	the universal family of $\Gamma$-equivariant $(\mathbb{C}^* \times {\rm Aut}(\mathfrak{g}))$-Higgs bundles in this case.
	\item Given a $(\mathbb{C}^* \times {\rm Aut}(\mathfrak{g}))$-bundle $F$, if the associated line bundle $F(\mathbb{C}^*)$ is trivial, the $(\mathbb{C}^* \times {\rm Aut}(\mathfrak{g}))$-bundle $F$ can be reduced to an ${\rm Aut}(\mathfrak{g})$-bundle. By the universal property of the Picard scheme ${\rm Pic}(Y)$, the associated family
	\begin{align*}
	\mathscr{E}_1(\mathbb{C}^*) \rightarrow \mathcal{R}_1 \times Y
	\end{align*}
	corresponds to a morphism $f: \mathcal{R}_1 \rightarrow {\rm Pic}(Y)$. Let $\mathcal{R}'_1=f^{-1}([\mathcal{O}_X])$. Then, the family
	\begin{align*}
	\mathscr{E}'_1 := \mathscr{E}_1|_{\mathcal{R}'_1} \rightarrow \mathcal{R}'_1 \times Y
	\end{align*}
	is a ${\rm GL}(V)$-universal family for $\Gamma$-equivariant logarithmic ${\rm Aut}(\mathfrak{g})$-Higgs bundles, of which the associated $\Gamma$-equivariant logarithmic ${\rm GL}(\mathfrak{g})$-Higgs bundles are semistable.
	\item Note that ${\rm Ad}(G)=G/Z \hookrightarrow {\rm Aut}(\mathfrak{g})$ is injective. By Lemma \ref{708}, the functor $\widetilde{\Gamma}({\rm Ad}, \mathscr{E}'_1)$ is representable. Let
	\begin{align*}
	\mathscr{E}_2 \rightarrow \mathcal{R}_2 \times Y
	\end{align*}
	be the ${\rm GL}(V)$-universal family of $\Gamma$-equivariant logarithmic $G/Z$-Higgs bundles, of which the associated $\Gamma$-equivariant logarithmic ${\rm GL}(\mathfrak{g})$-Higgs bundles are semistable.
	\item In this step, we will construct a universal family for $\Gamma$-equivariant logarithmic $(G/G'\times G/Z)$-Higgs bundles, where $G'$ is the derived group. Since $G$ is reductive, $G/G'$ is a torus. We assume $G/G' \cong (\mathbb{C}^{*})^l$. It is well-known that a $\mathbb{C}^*$-bundle is a line bundle, and ${\rm Pic}(Y)$ classifies all line bundles on $Y$. Therefore, $\prod^l {\rm Pic}(Y)$ parameterizes all $G/G'$-bundles. Denote by $\mathscr{P} \rightarrow {\rm Pic}(Y)$ the Poincar\'e bundle. We consider the following family
	\begin{align*}
	(\underbrace{\mathscr{P} \times_Y \dots \times_Y \mathscr{P}}_l) \times_{Y} \mathscr{E}_2 \rightarrow (\prod^l {\rm Pic}(Y) \times \mathcal{R}_2) \times Y
	\end{align*}
	of $(G/G' \times G/Z)$-bundles. We define
	\begin{align*}
	\mathcal{E}'_2:= (\underbrace{\mathscr{P} \times_Y \dots \times_Y \mathscr{P}}_l) \times_{Y} \mathscr{E}_2
	\end{align*}
	and
	\begin{align*}
	\mathcal{R}'_2:= (\prod^l {\rm Pic}(Y) \times \mathcal{R}_2).
	\end{align*}
	Then, $\mathscr{E}'_2$ is a ${\rm GL}(V)$-universal family of $\Gamma$-equivariant logarithmic $(G/G' \times G/Z)$-Higgs bundles.
	\item Finally, we consider the natural projection $\rho: G \rightarrow G/G' \times G/Z$. The functor $\widetilde{\Gamma}(\rho,\mathscr{E}'_2)$ is representable by a scheme $\mathcal{R}_3$ (see \cite[Lemma 4.15.1]{Rama19962}). Denote by $\mathscr{E}_3 \rightarrow \mathcal{R}_3 \times Y$ the ${\rm GL}(V)$-universal family of $(\Gamma,G)$-Higgs bundles, of which the associated $\Gamma$-equivariant logarithmic ${\rm GL}(\mathfrak{g})$-Higgs bundles are semistable. Therefore, the scheme $\mathcal{R}_3$ parameterizes $R$-semistable logarithmic $(\Gamma,G)$-Higgs bundles on $Y$.
\end{enumerate}

The above discussion gives the following proposition.
\begin{prop}
	There is a quasi-projective coarse moduli space $\mathfrak{M}^{Rss}_{H,\boldsymbol\theta}([Y \backslash \Gamma],G)$ of $R$-semistable logarithmic $(\Gamma,G)$-Higgs bundles on $Y$.
\end{prop}

As we explained at the beginning of this section, this proposition then implies that there exists a moduli space of $R$-semistable logahoric $\mathcal{G}_{\boldsymbol\theta}$-Higgs torsors on $X$. If we fix a particular element in $G/G' \cong (\mathbb{C}^*)^l$, which corresponds to the given topological data $\mu$ (see Remark \ref{rem adj bundle}), then we obtain a quasi-projective scheme $\mathfrak{M}^{Rss}_{H,\boldsymbol\theta}([Y \backslash \Gamma],G,\mu)$ as the coarse moduli space in Theorem \ref{701}.

\begin{rem}\label{moduli of local system}
We briefly discuss the construction of the moduli space of logahoric connections, where \emph{logahoric} is a blend of the words \emph{logarithmic} and \emph{parahoric}. Let $E$ be a parahoric $\mathcal{G}_{\boldsymbol\theta}$-torsor. A \emph{logarithmic connection} on $E$ is a connection $\nabla: \mathcal{O}_E \rightarrow \mathcal{O}_E \otimes K_X(D)$ such that the diagram 
\begin{center}
\begin{tikzcd}
	\mathcal{O}_E \arrow[rr,"\nabla"] \arrow[d,"a"] & & \mathcal{O}_E \otimes K_X(D) \arrow[d,"a \otimes 1"] \\
	\mathcal{O}_E \otimes \mathcal{O}_{\mathcal{G}_{\boldsymbol\theta}} \arrow[rr,"\nabla \otimes 1 + 1 \otimes \nabla_{\mathcal{G}_{\boldsymbol\theta}}"] & & (\mathcal{O}_E \otimes \mathcal{O}_{\mathcal{G}_{\boldsymbol\theta}}) \otimes K_X(D)
\end{tikzcd}
\end{center}
commutes, where $a$ is the co-multiplication map and $\nabla_{\mathcal{G}_{\boldsymbol\theta}}$ is the canonical connection on $\mathcal{O}_{\mathcal{G}_{\boldsymbol\theta}}$. The connection $\nabla$ gives a $D_X$-scheme structure on $E$ (see \cite[Appendix]{Chenzhu1}). A pair $(E,\nabla)$ is called a \emph{logahoric connection}. Similar to logahoric $\mathcal{G}_{\boldsymbol\theta}$-Higgs torsors, logahoric connections on $X$ correspond to $(\Gamma,G)$-bundles on $Y$ with a logarithmic connection (see \cite{BBP}). More precisely, a $\Gamma$-equivariant $G$-bundle with a logarithmic connection is a pair $(F,\nabla)$, where $F$ is a $(\Gamma,G)$-bundle on $Y$ and $\nabla$ is a $\Gamma$-equivariant logarithmic connection on $F$. Therefore, it is equivalent to construct the moduli space of $R$-semistable $\Gamma$-equivariant $G$-bundles $F$ with an equivariant logarithmic connection $\nabla$ on the stack $[Y/\Gamma]$, where $R$-semistability of the pair $(F,\nabla)$ can be defined in a similar way as in \S\ref{sect_stab_cond}. Note that the pair $(F,\nabla)$ can be regarded as a special case of a sheaf of graded algebras (see \cite{Simp2,Sun202003}). Therefore, the moduli space of semistable $\Gamma$-equivariant $G$-bundles with a logarithmic connection on $[Y/\Gamma]$ exists (see again \cite{Simp2,Sun202003}), and analogously to the construction provided in this section, one gets the construction of the moduli space of $R$-semistable logahoric connections.
\end{rem}
	
\begin{rem}\label{irrational} One gets the same moduli space when considering irrational weights. By \cite[Proposition 7.3]{BBP}, given an irrational weight $\theta \in Y(T) \otimes_{\mathbb{Z}} \mathbb{R}$, there exists a rational weight $\theta'$ in the same facet as the given weight $\theta$, and then the irrational case can be reduced to the rational case. This result can be analogously extended to the case when considering nonzero Higgs fields too.
\end{rem}

\section{Poisson Structure on the Moduli Space of Logahoric Higgs Torsors}

In this section, we will construct algebraically a Poisson structure on the moduli space of logahoric Higgs torsors for reductive groups $G$. This algebraic construction comes from the consideration of Lie algebroids, as in \cite{KSZ4,LoMa}; we next review this technique below.

By Theorem \ref{701}, we have an isomorphism
\begin{align*}
\mathfrak{M}^{Rss}_H(X,\mathcal{G}_{\boldsymbol\theta},\mu) \cong \mathfrak{M}^{Rss}_{H,\boldsymbol\theta}([Y/\Gamma],G,\mu).
\end{align*}
The authors studied Poisson structures on the moduli space of Higgs bundles on stacky curves in \cite{KSZ4}, and we use a similar approach to construct the Poisson structure here.

Before we demonstrate the construction of the Poisson structure on $\mathfrak{M}^{Rss}_{H,\boldsymbol\theta}([Y/\Gamma],G,\mu)$, we first review some results on Lie algebroids and Poisson structures. Let $\mathfrak{M}$ be a projective (or quasi-projective) scheme over $\mathbb{C}$ together with a proper and free group action $K \times \mathfrak{M} \rightarrow \mathfrak{M}$. Then, we have a natural projection $\pi: \mathfrak{M} \rightarrow \mathfrak{M}/K$, which induces
\begin{align*}
0 \rightarrow T_{\rm orb} \mathfrak{M} \rightarrow T \mathfrak{M} \rightarrow \pi^* T (\mathfrak{M}/K) \rightarrow 0.
\end{align*}
This exact sequence gives us a natural surjective morphism $T\mathfrak{M}/K \rightarrow T(\mathfrak{M}/K)$, which is the \emph{anchor map}. Then, we have
\begin{align*}
0 \rightarrow {\rm Ad}(\mathfrak{M}) \rightarrow T\mathfrak{M}/K \rightarrow T(\mathfrak{M}/K) \rightarrow 0,
\end{align*}
which is known as the \emph{Atiyah sequence}. The Atiyah sequence induces a \emph{Lie algebroid} structure on $T\mathfrak{M}/K$. By \cite[Theorem 2.1.4]{Cour}, the total space $(T\mathfrak{M}/K)^*$ has a Poisson structure.

Let $\mathcal{X}=[Y/\Gamma]$, and denote by $X$ the coarse moduli space of $\mathcal{X}$ with the natural morphism $\pi: \mathcal{X} \rightarrow X$. For simplicity, we use the following notation in the sequel
\begin{itemize}
	\item $\mathfrak{M}_H(\mathcal{X})$: moduli space of $R$-semistable logarithmic $(\Gamma,G)$-Higgs bundles of type $\boldsymbol\theta$ and $\mu$ on $Y$;
	\item $\mathfrak{M}^0_H(\mathcal{X})$: moduli space of $R$-stable logarithmic $(\Gamma,G)$-Higgs bundles $(F,\phi)$ of type $\boldsymbol\theta$ and $\mu$ on $Y$ such that $F$ is $R$-stable;
	\item $\mathfrak{M}(\mathcal{X})$: moduli space of $R$-stable $(\Gamma,G)$-bundles of type $\boldsymbol\theta$ and $\mu$ on $Y$;
	\item $\mathfrak{M}(X)$: moduli space of $R$-stable $G$-bundles on $X$.
	\item $\mu_1: \mathfrak{M}(\mathcal{X}) \times \mathcal{X} \rightarrow \mathfrak{M}(\mathcal{X})$ and $\nu_1: \mathfrak{M}(\mathcal{X}) \times \mathcal{X} \rightarrow \mathcal{X}$ are natural projections.
\end{itemize}

\begin{center}
	\begin{tikzcd}
	&  \mathfrak{M}(\mathcal{X}) \times \mathcal{X} \arrow[dl,swap, "\mu_1"] \arrow[dr, "\nu_1"] & \\
	\mathfrak{M}(\mathcal{X})   & & \mathcal{X} 
	\end{tikzcd}
\end{center}

Given a $G$-bundle $F$ on $\mathcal{X}$, the $G$-bundle $\pi^* \pi_* F$ is a $\Gamma$-invariant $G$-bundle on $\mathcal{X}$. Then, we have a natural morphism
\begin{align*}
{\rm Ad}(F)\hookrightarrow  {\rm Ad}(\pi^* \pi_* F).
\end{align*}
For each $\theta_y \in \boldsymbol\theta$, one gets a parabolic subgroup $P_y \subseteq G$. Denote by $P_y=L_y N_y$ the Levi factorization. Then, we have
\begin{align*}
0 \rightarrow {\rm Ad}(F)  \rightarrow {\rm Ad}(\pi^* \pi_* F)\rightarrow \prod_{y \in R} \mathfrak{n}_y \otimes \mathcal{O}_{y} \rightarrow 0.
\end{align*}
Suppose that $F$ is $R$-stable. By the deformation theory of $G$-bundles, the tangent space $T\mathfrak{M}(\mathcal{X})$ at $[F]$ is given by
\begin{align*}
T_{[F]}\mathfrak{M}(\mathcal{X}) \cong H^1(\mathcal{X}, {\rm Ad}(F)).
\end{align*}
Similarly, we have
\begin{align*}
T_{[\pi_*F]}\mathfrak{M}(X) \cong H^1(X, {\rm Ad}(\pi_* F)) \cong H^1(\mathcal{X}, {\rm Ad}(\pi^* \pi_* F)).
\end{align*}
By Grothendieck duality on stacky curves (see \cite{Nir08}), we have
\begin{align*}
T^*_{[F]}\mathfrak{M}(\mathcal{X}) \cong H^0(\mathcal{X}, {\rm Ad}(F) \otimes \omega_{\mathcal{X}}).
\end{align*}
Taking an element $\phi \in H^0(\mathcal{X}, {\rm Ad}(F) \otimes \omega_{\mathcal{X}})$, we get the following isomorphism
\begin{align*}
T_{\phi} T^*_{[F]}\mathfrak{M}(\mathcal{X}) \cong T_{(F,\phi)} \mathfrak{M}_H(\mathcal{X}),
\end{align*}
where $T_{(F,\phi)} \mathfrak{M}_H(\mathcal{X}) \cong \mathbb{H}^1({\rm Ad}(F) \xrightarrow{{\rm Ad}(\phi)} {\rm Ad}(F)\otimes \omega_{\mathcal{X}})$ (see \cite[\S 2.3]{GGR}).

Now let $\mathscr{F}$ and $\mathscr{E}$ be universal families on $\mathfrak{M}(\mathcal{X})$ and $\mathfrak{M}(X)$ respectively. The above discussion gives the following short exact sequence
\begin{align*}
0 \rightarrow {\rm Ad}(\mathscr{F}) \rightarrow {\rm Ad}(\mathscr{E}) \rightarrow \prod_{y \in R} \mathfrak{n}_y \otimes \mathcal{O}_{\nu_1^{-1}(y)} \rightarrow 0,
\end{align*}
which induces a short exact sequence
\begin{align*}
0 \rightarrow \mathscr{A}d \rightarrow R^1(\mu_1)_* {\rm Ad}(\mathscr{F}) \rightarrow R^1(\mu_1)_* {\rm Ad}(\mathscr{E}) \rightarrow 0.
\end{align*}
This sequence is an Atiyah sequence and so we have a Lie algebroid structure on $R^1(\mu_1)_* {\rm Ad}(\mathscr{E})$. Note that
\begin{align*}
R^1(\mu_1)_* {\rm Ad}(\mathscr{F}) \cong T \mathfrak{M}(\mathcal{X}) \quad \text{and} \quad  R^1(\mu_1)_* {\rm Ad}(\mathscr{E}) \cong T \mathfrak{M}(X).
\end{align*}
Then, $(T \mathfrak{M}(\mathcal{X}))^*$ has a Poisson structure. Moreover, it holds that
\begin{align*}
T(R^1(\mu_1)_* {\rm Ad}(\mathscr{F}))^* \cong T \mathfrak{M}_H(\mathcal{X}),
\end{align*}
thus by \cite[Theorem 2.1.4]{Cour} the moduli space $\mathfrak{M}^0_H(\mathcal{X})$ is equipped with a Poisson structure. Since $\mathfrak{M}^0_H(\mathcal{X})$ is dense in $\mathfrak{M}_H(\mathcal{X})$, there exists a Poisson structure on $\mathfrak{M}_H(\mathcal{X})$. We have proven the following:

\begin{prop}\label{801}
	There exists an algebraic Poisson structure on the moduli space $\mathfrak{M}^{Rss}_H(X,\mathcal{G}_{\boldsymbol\theta},\mu)$ of $R$-semistable logahoric $\mathcal{G}_{\boldsymbol\theta}$-Higgs torsors on $X$.
\end{prop}

\vspace{2mm}
\textbf{Acknowledgments}.
We are the most grateful to Philip Boalch for a series of constructive queries on a first draft of this article which have improved its content. We would also like to thank Pengfei Huang for his interest in this project and for useful discussions, as well as an anonymous referee for a careful reading of the manuscript and important remarks. G. Kydonakis is much obliged to the Max-Planck-Institut f\"{u}r Mathematik in Bonn for its hospitality and support during the production of this article. H. Sun is partially supported by National
Key R$\&$D Program of China (No. 2022YFA1006600) and NSFC (No. 12101243).

\bigskip
\noindent\small{\textsc{Alexander von Humboldt-Stiftung \& Universit\"{a}t Heidelberg}\\
	Mathematisches Institut Universit\"{a}t Heidelberg, Im Neuenheimer Feld 205, Heidelberg 69120, Germany}\\
\emph{E-mail address}:  \texttt{gkydonakis@mathi.uni-heidelberg.de}

\bigskip
\noindent\small{\textsc{Department of Mathematics, South China University of Technology}\\
	381 Wushan Rd, Tianhe Qu, Guangzhou, Guangdong, China}\\
\emph{E-mail address}:  \texttt{hsun71275@scut.edu.cn}

\bigskip
\noindent\small{\textsc{Department of Mathematics, University of Illinois at Urbana-Champaign}\\
	1409 W. Green St, Urbana, IL 61801, USA}\\
\emph{E-mail address}: \texttt{zhaolutian1994@gmail.com}

\end{document}